\documentclass[11pt]{amsart}

\usepackage{amsfonts, amstext, amsmath, amsthm, amscd, amssymb, upgreek}
\usepackage[pdftex]{graphicx, color}
\usepackage[dvipsnames]{xcolor}
\usepackage[all,cmtip]{xy} 
\usepackage{tikz}
\usepackage{import}
\usepackage[hidelinks,pagebackref]{hyperref}
\usepackage{subfigure, wrapfig, overpic}

\let\oldmarginpar\marginpar
\renewcommand\marginpar[1]{\oldmarginpar[\raggedleft\footnotesize #1]%
{\raggedright\footnotesize #1}}

\usepackage{marginnote}
\long\def\@savemarbox#1#2{\global\setbox#1\vtop{\hsize\marginparwidth 
  \@parboxrestore\tiny\raggedright #2}}
\renewcommand*{\backref}[1]{}
\renewcommand*{\backrefalt}[4]{
  \ifcase #1
  [No citations.]
  \or [#2]
  \else [#2]
  \fi }

\AtBeginDocument{
   \def\MR#1{}
}

\theoremstyle{plain}
\newtheorem{theorem}{Theorem}[section]
\newtheorem{corollary}[theorem]{Corollary}
\newtheorem{lemma}[theorem]{Lemma}

\newtheorem{proposition}[theorem]{Proposition}

\newtheorem{conjecture}[theorem]{Conjecture}

\newtheorem*{namedtheorem}{\theoremname}
\newcommand{\theoremname}{testing}
\newenvironment{named}[1]{\renewcommand{\theoremname}{#1}\begin{namedtheorem}}{\end{namedtheorem}}

\theoremstyle{definition}

\newtheorem{definition}[theorem]{Definition}
\newtheorem{question}[theorem]{Question}
\newtheorem{remark}[theorem]{Remark}
\newtheorem{example}[theorem]{Example}

\newcommand{\refthm}[1]{Theorem~\ref{Thm:#1}}
\newcommand{\reflem}[1]{Lemma~\ref{Lem:#1}}
\newcommand{\refprop}[1]{Proposition~\ref{Prop:#1}}

\newcommand{\refconj}[1]{Conjecture~\ref{Conj:#1}}

\newcommand{\refitm}[1]{\eqref{Itm:#1}}
\newcommand{\refdef}[1]{Definition~\ref{Def:#1}}
\newcommand{\refsec}[1]{Section~\ref{Sec:#1}}
\newcommand{\reffig}[1]{Figure~\ref{Fig:#1}}
\newcommand{\refques}[1]{Question~\ref{Ques:#1}}

\newcommand{\R}{\mathbb{R}}

\newcommand{\CC}{\mathbb{C}}

\newcommand{\HH}{\mathbb{H}}

\renewcommand{\P}{\overline P}
\newcommand{\F}{\mathcal F}
\newcommand{\W}{\mathcal W}
\newcommand{\vol}{{\rm vol}}
\newcommand{\guts}{{\rm guts}}
\newcommand{\cut}{{\backslash \backslash}}
\newcommand{\bdy}{\partial}
\newcommand{\voct}{{v_{\rm oct}}}
\newcommand{\vtet}{{v_{\rm tet}}}

\newcommand{\toF}{{\overset{F}{\longrightarrow}}}

\newcommand{\volp}{{\rm vol}^{\perp}}
\newcommand{\half}{\frac{1}{2}}
\newcommand{\Q}{\mathcal Q}
\newcommand{\T}{\mathcal T}

\def\split{\setminus \!\! \setminus}

\title[Right-angled polyhedra and alternating links]{Right-angled polyhedra and alternating links}
\author[A.\ Champanerkar]{Abhijit Champanerkar}
\author[I.\ Kofman]{Ilya Kofman}
\address{Department of Mathematics, College of Staten Island \& The Graduate Center, City University of New York, New York, NY}
\email{abhijit@math.csi.cuny.edu, ikofman@math.csi.cuny.edu}
\author[J. \ Purcell]{Jessica S.\ Purcell}
\address{School of Mathematics, 9 Rainforest Walk, Monash University, Victoria 3800, Australia}
\email{jessica.purcell@monash.edu}

\begin{document}

\begin{abstract}
To any prime alternating link, we associate a collection of hyperbolic
right-angled ideal polyhedra by relating geometric, topological and
combinatorial methods to decompose the link complement.  The sum of
the hyperbolic volumes of these polyhedra is a new geometric link
invariant, which we call the right-angled volume of the alternating
link.  We give an explicit procedure to compute the right-angled
volume from any alternating link diagram, and prove that it is a new
lower bound for the hyperbolic volume of the link.
\end{abstract}

\maketitle

\section{Introduction}
Right-angled structures have featured in many striking results in
$3$-manifold geometry, topology and group theory.  Alternating knot
complements decompose into pieces with essentially the same
combinatorics as the alternating diagram, so the hyperbolic geometry
of alternating knots is closely related to their diagrams.  Therefore, it is
natural to ask: Which alternating links are right-angled?  In other
words, when does the link complement with the complete hyperbolic
structure admit a decomposition into ideal hyperbolic right-angled
polyhedra?  Surprisingly, besides the Whitehead and Borromean links,
only two other alternating links are known to be right-angled \cite{Gan}.
In Section~\ref{sec:right-angled-knots} below, we conjecture that,
whether alternating or not, there does not exist a right-angled knot.

However, we can obtain right-angled structures from alternating
diagrams, which do not give the complete hyperbolic structure, but
nevertheless provide useful and computable geometric link invariants.
These turn out to be related to previous geometric constructions
obtained from alternating diagrams, and below we relate these
constructions to each other in new ways.

Specifically, in this paper we associate a set of hyperbolic
right-angled ideal polyhedra to any reduced, prime, alternating link
diagram.  We prove that these polyhedra can be described equivalently
from the following geometric, topological and combinatorial
perspectives.  This equivalence implies that such a set of
right-angled polyhedra is a link invariant, whose volume sum we call
the {\em right-angled volume} $\volp(K)$ of the alternating link $K$.
We prove that this new geometric link invariant is a lower bound for
the hyperbolic volume of the link, and is asymptotically sharp for
certain sequences of knots and links.

\subsection*{Geometry}
The {\em guts} of a $3$--manifold cut along an essential surface is
the union of all the hyperbolic pieces in its JSJ-decomposition.
Lackenby~\cite{Lackenby}, building on work of Thurston~\cite{tnotes},
Agol~\cite{Agol:Guts} and Menasco~\cite{menasco:polyhedra}, described
such a geometric decomposition of the complement of any alternating
link by cutting along its two checkerboard surfaces.
In~\cite{CKPgmax}, we determined explicitly the guts of the manifolds
obtained from alternating links with certain extra hypotheses, cut
along \emph{both} checkerboard surfaces.  In \refsec{geometry}, we
again analyse the guts of these manifolds, but we remove the extra
hypotheses on the alternating links, and we construct the associated
hyperbolic right-angled {\em guts polyhedra}.

\subsection*{Topology}
Bonahon-Siebenmann~\cite{BonahonSiebenmann}, building on work of
Conway~\cite{Conway} and Montesinos~\cite{Montesinos}, defined a {\em
  characteristic splitting} of any link diagram along Conway spheres
into arborescent and non-arborescent parts.  In the double branched
cover of the link, the arborsescent part is covered by a graph
manifold, and the non-arborsescent part by a hyperbolic manifold.
Menasco and Thistlethwaite~\cite{menasco84, menasco-thist:alternating,
  Thistlethwaite_JKTR, Thistlethwaite:Algebraic} showed that, up to
flypes, an alternating diagram can be decomposed only in limited ways
along invariant Conway spheres into alternating tangles.
Thistlethwaite~\cite{Thistlethwaite:Algebraic} used such tangles to
completely describe the characteristic splitting of any alternating
link diagram.  In \refsec{topology}, we use Thistlethwaite's results
to build {\em tangle polyhedra} associated with the non-arborescent
part of an alternating link.

\subsection*{Combinatorics}
Right-angled polyhedra are natural hyperbolic ``bricks'' which have
been used to construct hyperbolic 3-manifolds.  In 1931,
L\"obell~\cite{lobell} constructed the first example of a closed
orientable hyperbolic 3-manifold by gluing eight copies of the
right-angled 14-hedron.  Andreev's Theorem (\refthm{Andreev} below)
implies that, up to isometry, a hyperbolic right-angled ideal
polyhedron is uniquely determined by its combinatorial type.  In
\refsec{combinatorics}, we define {\rm rational reduction} of an
alternating link diagram, and then determine its diagrammatic
splitting into {\em Andreev polyhedra}.

\vspace{.1in}

One main result of this paper is the somewhat surprising fact that the three polyhedra, obtained from now-classic geometric, topological, and combinatorial methods, actually give the same link invariant for alternating links!

\begin{named}{\refthm{AndreevEqualsGuts}}
The Andreev polyhedra are identical to the tangle polyhedra and the guts polyhedra. 
\end{named}

\subsection*{Right-angled volume}
In \refsec{volume}, we define the right-angled volume $\volp(K)$ of an
alternating link $K$ as the sum of the hyperbolic volumes of these
right-angled polyhedra. It gives a new geometric link invariant for
alternating links.

\begin{named}{\refthm{GutsVolumeBound}}
  For any hyperbolic alternating link $L$ with hyperbolic volume $\vol(L)$,
  \[ \volp(L)\:\leq\:\vol(L). \]
\end{named}

Morover, we show the bound is asymptotically sharp: There exist many sequences of alternating links $K_n$ such that
\[ \lim_{n\to\infty}\frac{\volp(K_n)}{\vol(K_n)}=1. \]

To compare $\volp(L)$ with other volume bounds, we exhibit examples
for which $\volp(L)$ beats the best previous lower bounds,
from~\cite{AgolStormThurston, Lackenby}, by a factor of two. However,
for any Montesinos link, which may have arbitrarily large volume,
$\volp(L)=0$. Thus, as a lower volume bound, $\volp(L)$ should be used
together with other bounds. We discuss this further in \refsec{Discussion}.

\subsection*{Acknowledgements}
We thank David Futer, Anastasiia Tsvietkova, and Francis Bonahon for
useful discussions.  The first two authors acknowledge the support of
the Simons Foundation and PSC-CUNY.  The third author acknowledges the
support of the Australian Research Council.

\section{Geometry}\label{Sec:geometry}

\subsection{Background on guts}
By work of Menasco, if a link has a reduced, prime, alternating diagram that is not the diagram of a $(2,q)$-torus link, then the link complement is hyperbolic. However, it can be difficult to determine geometric information directly from a diagram.

We will cut a $3$--manifold along an essential surface and consider
its JSJ-decomposition \cite{Jaco, JacoShalen, Johannson}, which cuts
the $3$--manifold into components consisting of $I$-bundles, Seifert
fibered pieces, and \emph{guts}. To describe the guts, we need to set
up some definitions and notation. We will mainly consider 3-manifolds
$M$ that admit a finite volume hyperbolic structure. We may view such
a manifold as the interior of a compact manifold $\overline{M}$ with
torus boundary components. The fact that $M$ is hyperbolic means it
admits a hyperbolic structure $M \cong \HH^3/\Gamma$, where $\Gamma$
is a discrete subgroup of $\rm{PSL}(2,\CC)$. Under this structure, any
closed curve in a neighbourhood of $\bdy\overline{M}$ is isotopic to a
\emph{parabolic} element of $\Gamma$.

We view $M$ as an open manifold without boundary, but
at times it will be more convenient to consider the compact manifold $\overline{M}$. In the case of a link complement, $M=S^3-L$, the compact manifold $\overline{M}$ is homeomorphic to $S^3-N(L)$, where $N(L)$ denotes an open regular neighbourhood of $L$ in $S^3$. Then $\bdy\overline{M}$ is a collection of tori. These form the \emph{parabolic locus} of $\overline{M}$. More generally, the parabolic locus $\lambda$ of a compact 3-manifold $\overline{M}$ will consist of annuli and tori in $\bdy \overline{M}$. When we want to carefully keep track of the parabolic locus, we write the manifold as a pair $(\overline{M},\lambda)$. For our link example, we have the pair $(S^3-N(L),\bdy N(L))$.

Now, we wish to cut a 3-manifold $M$ along an essential surface $S$. When we view $M$ as an open manifold, then $S$ will be a properly embedded open surface, homeomorphic to the interior of a compact surface $\overline{S}$ with boundary on $\bdy \overline{M}$. Let $M\cut S$ denote the closure (in $M$) of the manifold obtained by removing a regular neighbourhood of $S$ from $M$.  The boundary of $M\cut S$ is homeomorphic to $\widetilde{S} = \bdy N(S)$, the double cover of $S$.

On the other hand, we may also express this information in terms of a pair. If $\overline{M}$ has parabolic locus $\bdy\overline{M}$ consisting of tori, then we express the cut manifold as a pair $(\overline{M}\cut \overline{S}, \bdy(\overline{M})\cut\bdy\overline{S})$. Note that the parabolic locus will now include annular components.

A pair $(\overline{M},\lambda)$ is called a \emph{pared acylindrical 3-manifold} if $\overline{M}$ is a compact, irreducible, atoroidal manifold with boundary (such as $S^3-N(L)$ or $\overline{M}\cut\overline{S}$) and $\lambda\subset \bdy\overline{M}$ is a union of incompressible annuli and tori, such that every map $(S^1\times I, S^1\times\bdy I) \to (\overline{M}, \bdy\overline{M}-\lambda)$ that is $\pi_1$-injective deforms as a map of pairs into $\lambda$. Denote $\bdy\overline{M}-\lambda$ by $\bdy_0\overline{M}$. Thurston showed that a pared acylindrical 3-manifold admits a hyperbolic metric with totally geodesic boundary $\bdy_0 M$ and parabolic locus $\lambda$ \cite{Morgan}.

For $M$ an open 3-manifold, and $S$ an essential surface, the \emph{guts} of $M\cut S$, denoted $\guts(M\cut S)$, is the union of all components in the JSJ-decomposition of $M\cut S$ that admit a hyperbolic structure. In terms of the notation of pared manifolds, let $(\overline{M}, \bdy\overline{M})$ denote the pair corresponding to $M$. Let $A$ denote the union of essential annuli in $M\cut S$. Let $M_1$ be a hyperbolic component of $(M\cut S)\cut A$, so $M_1$ is a component of $\guts(M\cut S)$. It is associated to a pair $(\overline{M}_1, \lambda_1)$, where $\overline{M}_1$ is the union of $M_1$ along with its boundary as a subset of $\overline{M}$, and $\lambda_1$ consists of $\bdy\overline{M}\cap \bdy M_1$ as well as any component of $\widetilde{A} \cap M_1$. That is, when we take the JSJ-decomposition of a 3-manifold, all essential annuli that we cut along to form the decomposition become part of the parabolic locus of the guts. 

\subsection{Checkerboard decomposition}
Suppose $L$ is a link that admits a reduced, prime, alternating
diagram.  Its two checkerboard surfaces are essential~\cite{Aumann,Lackenby},
and so we may follow the procedure outlined above and cut along a checkerboard surface $S$, obtaining the guts of $(S^3-L)\cut S$. Indeed, we will use a doubling procedure to cut along both checkerboard surfaces.

Let $L$ be a link with reduced, prime, alternating diagram, and associated checkerboard surfaces $B$ and $R$. It is well-known that cutting $S^3-L$ along both $B$ and $R$ simultaneously decomposes it into two identical (topological) ideal polyhedra~\cite{menasco:polyhedra}. For an alternating link, each of the two ideal polyhedra is obtained by taking edges and ideal vertices
corresponding to the diagram graph of the link. Call one of these
ideal polyhedra the \emph{checkerboard polyhedron} associated to the
link diagram.

Instead of cutting along both surfaces simultaneously, we cut along the two surfaces separately and consider the guts. Let $M_B$ denote the 3-manifold consisting of the guts of $(S^3-L)\cut B$, and $M_R$ the guts of $(S^3-L)\cut R$; i.e.\
\[ M_B = \guts((S^3-L)\cut B); \quad M_R = \guts((S^3-L)\cut R).\]
The boundary $\bdy_0 M_B$ consists of $M_B\cap \widetilde{B}$, and similarly $\bdy_0 M_R$ consists of $M_R\cap\widetilde{R}$. The parabolic locus of each consists of remnants of the link and essential annuli that we cut along to form the guts. 

Let $D(M_B)$ denote the double of $M_B$ along the surface $\bdy_0M_B$. This manifold admits a hyperbolic structure in which $\bdy_0 M_B$ is totally geodesic. Similarly for $D(M_R)$. 

In~\cite[Lemma~4.8]{CKPgmax}, we showed that the surface $R \cap M_B$ doubles to give an essential surface $DR$ in $D(M_B)$. Similarly, $B\cap M_R$ doubles to give an essential surface $DB$ in $D(M_R)$. Thus we may cut along these surfaces, and find the guts of the resulting pieces.  That is, 
consider the manifold $\guts(D(M_B)\cut DR)$. This is a 3-manifold with boundary consisting of $\widetilde{DR} \cap \guts(D(M_B)\cut DR)$. Its double $D(\guts(D(M_B)\cut DR))$ therefore admits a hyperbolic structure in which $\widetilde{DR} \cap \guts(D(M_B)\cut DR)$ is totally geodesic.

In \cite{CKPgmax}, we showed that under certain hypotheses on $L$, $\guts(D(M_B)\cut DR)$ is the entire manifold $D(M_B)\cut DR$. In this case, the double $D(\guts(D(M_B)\cut DR))$ is built by gluing eight copies of the original checkerboard polyhedron obtained by cutting $S^3-L$ along $B$ and $R$. We now drop the restrictions on $L$ from \cite{CKPgmax}.

Let $D(D((S^3-L)\cut B)\cut DR)$ denote the manifold obtained first by cutting along $B$, then doubling, then cutting along $DR$ and doubling.
We will show:

\begin{proposition}\label{Prop:GutsHomeomorphisms}
The manifold $D(\guts(D(M_B)\cut DR))$ is homeomorphic to:
\begin{itemize}
\item The guts of the manifold $D(D((S^3-L)\cut B)\cut DR)$,
\item the guts of the manifold $D(D((S^3-L)\cut R)\cut DB)$, 
\item $D(\guts(D(M_R)\cut DB))$.
\end{itemize}
Moreover, all four manifolds are built by gluing eight copies of a collection of polyhedra, obtained by cutting the checkerboard polyhedron along normal squares and collapsing each normal square boundary to an ideal vertex. 
\end{proposition}

A \emph{normal square} is a disk properly embedded in the checkerboard polyhedron that meets exactly four faces of the polyhedron in \emph{normal form} --- basically transversely and without backtracking; see for example~\cite{Lackenby}, or~\cite[Definition~3.15]{fkp:gutsjp} for a precise definition of normal form, or \reflem{AnnulusSquares} below for a description of normal squares in a checkerboard polyhedron. Note that we will use the term \emph{square} below to refer to (normal) squares in a polyhedron as well as to corresponding curves in a diagram, i.e.\ curves that meet the diagram exactly four times, using the correspondance between the diagram and the checkerboard polyhedra noted above. 

\begin{definition}[Guts polyhedra]\label{Def:GutsPolyhedra}
  Let $L$ be a link with a reduced, prime, alternating diagram. The \emph{guts polyhedra} associated to the diagram are the ideal polyhedra obtained from \refprop{GutsHomeomorphisms} by taking one of the eight copies of polyhedra making up the guts of
  \[ D(D((S^3-L)\cut B)\cut DR). \]
  Alternatively, it is obtained by cutting the checkerboard polyhedron along normal squares required for that proposition, and discarding Seifert fibered or $I$-bundle components. 
\end{definition}

Our main interest in studying the guts polyhedra is the following result:

\begin{theorem}\label{Thm:RightAngledGuts}
The unique hyperbolic structure on $D(\guts(D(M_B)\cut DR))$ induces a hyperbolic structure on the guts polyhedra in which the red and blue faces are totally geodesic, with red faces meeting blue at right angles. 
\end{theorem}

The rest of this section gives the proofs of \refprop{GutsHomeomorphisms} and \refthm{RightAngledGuts}.

We will prove a sequence of results that will help us understand the form of the guts polyhedra, and how to identify them in a given link complement. First, we present an example. 

\begin{example}[Borromean rings]
  When $L$ is the standard reduced alternating diagram of the Borromean rings, then Thurston showed that the checkerboard polyhedron associated with $S^3-L$ is a regular ideal octahedron~\cite{tnotes}. Checkerboard color the faces blue and red. In this case, $(S^3-L)\cut B$ is obtained by cutting two copies of the regular ideal octahedron along blue faces. The result is hyperbolic with geodesic boundary, and so its guts, $M_B$ is built of two copies of the regular ideal octahedron. When we double, $D(M_B)$ is made of four copies of the regular ideal octahedron. Now cut along $DR$. This cuts along red faces, but again the manifold, built of four regular ideal octahedra, is hyperbolic with geodesic boundary. Hence its guts is a manifold made up of the four octahedra, and its double is made up of eight octahedra. Tracing back through the definition, in this case the guts polyhedra is the single regular ideal octahedron. It agrees with the checkerboard polyhedron.
\end{example}

In order to identify guts polyhedra, we need to identify the guts of manifolds obtained from the original link complement. In order to identify guts, we need to identify tori and annuli of the JSJ decomposition of the cut manifold. Because we are assuming we begin with a hyperbolic link, in fact there will be no essential tori, and we need only to idenfity essential annuli. The following lemma gives the relationship of the essential annuli to the normal squares of \refprop{GutsHomeomorphisms}.

\begin{lemma}\label{Lem:AnnulusSquares}
Let $M$ be an irreducible, boundary irreducible 3-manifold, $S$ an essential surface properly embedded in $M$, such that $M\cut S$ can be decomposed into a finite number of 4-valent checkerboard polyhedra with red and blue faces, where blue faces map to $\widetilde{S}$. Let $A$ be an essential annulus in $M\cut S$ with boundary components $\bdy A \subset \widetilde{S}$. Then $A$ can be isotoped into normal form with respect to the checkerboard polyhedra; i.e., 
  \begin{itemize}
  \item $A$ meets the polyhedra in disks.
  \item Each such disk is a square: it has exactly four sides running through four faces, with opposite sides in faces of the same colour, and it meets four edges and no vertices of the polyhedron.
  \item Each side of a square is an arc in a face of the polyhedron with endpoints on distinct edges.
  \end{itemize}
\end{lemma}

\begin{proof}
The fact that such an essential surface can be put into normal form is well-known; for example it is noted in \cite{Lackenby, fkp:gutsjp}. By an Euler characteristic argument, such an annulus decomposes into normal squares. 
\end{proof}

\begin{lemma}\label{Lem:GutsPolyhedra}
  For $L$ a link with a reduced, prime, alternating diagram, and notation as above, both the manifolds $D(M_B)$ and $D(\guts(D(M_B)\cut DR))$ have a decomposition into a finite collection of 4-valent ideal polyhedra that admit a checkerboard colouring, red and blue. Each polyhedron in the collection can be identified with a subset of the checkerboard polyhedron of the diagram of $L$, obtained by cutting the checkerboard polyhedron along a normal square, and then collapsing the normal square to a new ideal vertex. Each red (blue) face is a subset of a red (blue) face of the checkerboard polyhedron.
Finally, $D(\guts(D(M_B)\cut DR))$ is made up of eight copies of a finite collection of such polyhedra. 
\end{lemma}

\begin{proof}
We apply \reflem{AnnulusSquares} twice. If $A$ is an essential annulus in $(S^3-L)\cut B$, then $A$ meets the checkerboard polyhedra of $S^3-L$ in a collection of normal squares, by that lemma. Thus when we cut along $A$, we split the checkerboard polyhedra into new polyhedra, each a subset of the checkerboard polyhedra. Because $A$ becomes part of the parabolic locus in $M_B$, we then collapse the squares that came from $A$ into ideal vertices. Thus after the first step of taking the guts, we have split checkerboard polyhedra into new ideal polyhedra with a red--blue checkerboard colouring as described. Discard polyhedra that give $I$-bundle or Seifert fibered pieces, and double the result along the blue faces. This doubles the number of polyhedra, and gives the manifold $D(M_B)$. So $D(M_B)$ is built of 4-valent polyhedra that satisfy the conclusions of the first part of the lemma. The surface $DR$ is the image of the red faces of these polyhedra under the gluing.

For the second step, we cut along $DR\cap D(M_B)$, which is equivalent to cutting the 4-valent polyhedra of $D(M_B)$ along red faces. Then we find essential annuli in $D(M_B)\cut DR$. Any such annulus can be put into normal form with respect to the polyhedra for $D(M_B)$, and by \reflem{AnnulusSquares} it meets the polyhedra of $D(M_B)$ in squares. Again these cut the polyhedra into new checkerboard coloured polyhedra, and after collapsing the squares to ideal vertices we have 4-valent ideal polyhedra as claimed.

For the final statement, we count the number of copies of polyhedra. Cutting along normal squares to obtain $M_B$ and then doubling gives four copies of polyhedra obtained by cutting the original checkerboard polyhedron along normal squares. Cutting these along normal squares and then doubling gives eight copies of new polyhedra obtained from the original by cutting along normal squares.
\end{proof}

\subsection{Identifying the guts}

In the process of proving the rest of \refprop{GutsHomeomorphisms} and \refthm{RightAngledGuts}, we will explicitly determine the guts from a reduced, prime, alternating diagram of the link $L$. This will allow us to explicitly find the guts polyhedra. In particular, not every polyhedron obtained by cuting the checkerboard polyhedron along a normal square will be part of the guts. First, not every normal square is necessarily a square in the decomposition of an essential annulus into normal form as in \reflem{AnnulusSquares}. Second, some polyhedra resulting may be part of the $I$-bundle or Seifert fibered components of the cut manifold, not the guts.

The next lemma addresses the first point for the manifold $(S^3-L)\cut B$. We will call a normal square a {\em nontrivial square} if it bounds more than one vertex on each side.

\begin{lemma}\label{Lem:EssentialAnnuliStep1}
Let $L$ be a link with a reduced, prime, alternating diagram. 
The following nontrivial squares in the checkerboard polyhedra of $S^3-L$ give essential annuli in the manifold $(S^3-L)\cut B$:
\begin{enumerate}
\item A square bounding a string of red bigons, or
\item A cycle of fused units, i.e.\ a string of squares, each bounding a fused unit, as in \reffig{FusedUnit}.
\end{enumerate}
Conversely, any essential annulus for $(S^3-L)\cut B$ in normal form with respect to the checkerboard polyhedra decomposes into squares with one of the above two forms. Similarly for $(S^3-L)\cut R$, with red and blue surfaces swapped.
\end{lemma}

\begin{figure}
  \includegraphics{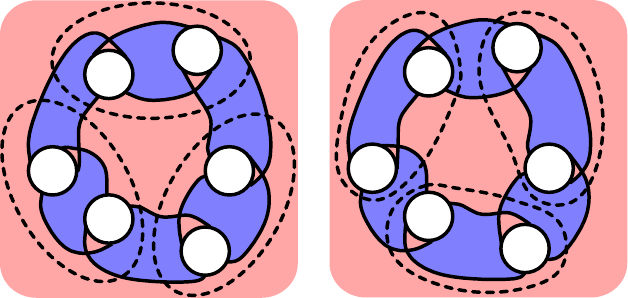}
  \hspace{.5in}
  \includegraphics{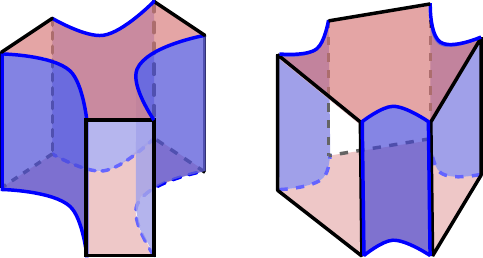}
  \caption{Left: A cycle of three fused units, with dashed squares in each polyhedron. Right: A Seifert fibered solid torus is built by gluing top faces to top faces, bottom to bottom (i.e.\ gluing red faces), with a half-turn on one side.}
  \label{Fig:FusedUnit}
\end{figure}

\begin{proof}
This is essentially due to Lackenby, with the proof contained in Section~5 of \cite{Lackenby}. We step through the argument.

The first case, of a square bounding a string of red bigons, is dealt with in the first two paragraphs of \cite[Section~5]{Lackenby}. Each red bigon is a disc with boundary consiting of four edges: two edges lie on blue faces, and two edges lie on vertices of the polyhedron. Each of these can be given a product structure $I\times(\mbox{arc of }\widetilde{B})$. The parabolic locus $P$ of $(S^3-L)\cut B$ also has a product structure of the form $I\times(\mbox{arc of }\widetilde{B})$, with the product structure of the red bigon matching that of the parabolic locus. Thus a neighbourhood of the union of the parabolic locus and red bigon faces cuts off an $I$-bundle component of $(S^3-L)\cut B$. Its boundary is an essential annulus in $(S^3-L)\cut B$. In normal form, this essential annulus will be made up exactly of squares bounding a string of red bigons (the boundary of a neighbourhood of red bigons) and trivial squares parallel to vertices. Each nontrivial square bounding red bigons appears in an essential annulus. 

In the second case, consider the dashed squares shown in \reffig{FusedUnit}, left, or more generally any number of these in a cycle. The gluing on the checkerboard polyhedra glues these along their red faces to the dashed squares shown second from left in \reffig{FusedUnit}, and the two collections of squares form an annulus. Together, these cut off a Seifert fibered solid torus, as shown on the right of \reffig{FusedUnit}. Thus this is an essential annulus. 

Conversely, in the proof of \cite[Theorem~14]{Lackenby}, Lackenby shows that if there are no red bigons in the diagram, and thus no essential annuli arising from an $I$-bundle determined by bigons as in case (1), then an essential annulus will be made up of fused units as on the left of \reffig{FusedUnit}; see also \cite[Figures~12,~13]{Lackenby}.
\end{proof}

\begin{lemma}\label{Lem:EssentialAnnuliStep2}
Let $L$ be a link with a reduced, prime, alternating diagram and associated checkerboard polyhedra. Then any square in the original checkerboard polyhedra for the link complement $S^3-L$ gives rise to an embedded annulus in $D(M_B)\cut DR$, and this annulus is essential in $D(M_B)\cut DR$ if and only if it is a nontrivial square.

Conversely, any embedded essential annulus in $D(M_B)\cut DR$ is obtained by a sequence of an even number of nontrivial squares in the original checkerboard polyhedra.
\end{lemma}

\begin{proof}
We prove the ``conversely'' statement first. Note it is basically contained in the proof of \cite[Lemma~4.10]{CKPgmax}, but we repeat the argument here.  By \reflem{AnnulusSquares}, an essential annulus is made up of squares that meet red and blue faces of the polyhedra of $D(M_B)\cut DR$. Since we glue by the identity on blue faces, the squares must glue together as shown on the left of \reffig{EssentialAnnuliStep2}, which is modified from \cite{CKPgmax}. In that figure, dashed lines indicate squares that lie in one copy of the polyhedron, while straight lines indicate squares in a second copy, glued to the first by the identity map on blue faces. Note the dashed squares and straight squares glue into an annulus. If there are only two squares, and each bounds a region containing a single ideal vertex as on the right of the figure, then each square making up the annulus is parallel to the ideal vertex shown. Note after gluing blue faces by the identity, the single ideal vertex becomes an annulus in the parabolic locus. The two squares parallel to that vertex will be parallel to the annulus, hence not essential. 

\begin{figure}
  \includegraphics{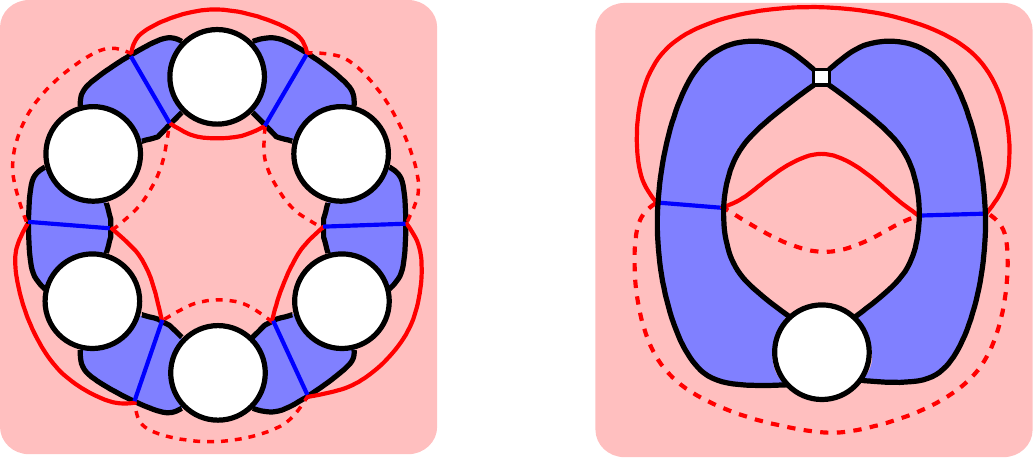}
  \caption{Left: The form of an essential annulus in $D(M_B)\cut DR$. Right: if the square encircles a single vertex the corresponding annulus is inessential.}
  \label{Fig:EssentialAnnuliStep2}
\end{figure}

To prove the first statement, given any square in the polyhedron that encircles more than one ideal vertex on each side, build an annulus in $D(M_B)\cut DR$ by taking one copy of the square in each of the two polyhedra that are glued, and gluing by the identity map on blue sides. Note if the square encircles a single ideal vertex, then the pair of squares glued in $D(M_B)\cut DR$ encircle the annulus in the parabolic locus $P$ that comes from the vertex; in particular by gluing the blue faces adjacent to the vertex by the identity map. This is not an essential annulus. 

Suppose the annulus is boundary compressible. Then a boundary compression disk has one arc of the boundary on the annulus, and another arc on the red surface. We may cut along the squares, shrink them to ideal vertices, and put the boundary compression disk into normal form with respect to these new polyhedra. If the boundary compression disk meets a blue face, then an outermost arc of intersection of the blue face and the compression disk must run from the ideal vertex corresponding to the square to an adjacent edge of the diagram. This contradicts the definition of normal. Thus the boundary compression disk does not meet a blue face, and thus lies in a single polyhedron. But this means a single arc of the boundary of the compression disk connects opposite sides of the square, red to red. This is impossible because the red faces on opposite sides are not connected (because the diagram is prime). Thus any such annulus is boundary incompressible. If not boundary parallel, then it is essential.
\end{proof}

In light of \reflem{EssentialAnnuliStep1} and especially \reflem{EssentialAnnuliStep2}, determining the guts is analogous to examining squares in the polyhedral decomposition of each manifold. Because the polyhedral decomposition comes from cutting the checkerboard polyhedra of an alternating diagram along squares, finding the guts amounts to analysing squares in the diagram graph of the alternating link.

The essential annuli that are most important in determining the guts are those that separate hyperbolic pieces from other hyperbolic pieces, or separate hyperbolic pieces from $I$-bundle or Seifert fibered pieces. By uniqueness of the JSJ-decomposition, these annuli are unique up to isotopy and pairwise disjoint. There may be additional essential annuli embedded in the Seifert fibered and $I$-bundle components of the cut manifolds, and these may not be disjoint from each other. Because we choose to keep the guts only, they do not affect our results. However, we do need to be able to recognise them and discard them.

\begin{lemma}\label{Lem:IBundlesBigonStrings}
Suppose $L$ is link with reduced, prime, alternating diagram and corresponding checkerboard polyhedra. Let $A$ denote a maximal collection of disjoint essential annuli in $D(M_B)\cut DR$. Let $C$ be a component of the polyhedra of $(D(M_B)\cut DR)\cut A$. Then $C$ gives rise to an $I$-bundle or Seifert fibered piece of $(D(M_B)\cut DR)$ if and only if $C$ has the combinatorics of the standard diagram of a $(2,q)$-torus link. 
That is, $C$ is not part of $\guts(D(M_B)\cut DR)$ exactly when the red or blue faces of $C$ form a chain of bigons. 
\end{lemma}

By ``maximal'' we mean a collection of disjoint essential annuli that is maximal in the sense that there is no other essential annulus that is pairwise disjoint from those already in the collection. 

\begin{proof}[Proof of \reflem{IBundlesBigonStrings}]
Suppose $C$ has the combinatorics of a $(2,q)$-torus link. Suppose first that the bigon faces of $C$ are all coloured red. Then $C$ can be given the structure of an $I$-bundle: it is homeomorphic to $I\times B_1$ where $B_1$ is one of the two blue faces of $C$, with each red face of the form $I\times(\mbox{arc of }\bdy B_1)$. The parabolic locus is also of the form $I\times(\mbox{arc of }\bdy B_1)$. To form $D(M_B) \cut DR$, double along the two blue faces. The $I$-bundle is doubled along $\bdy I \times B_1$, forming a Seifert fibered solid torus $S^1\times B_1$.

Now suppose that the bigon faces of $C$ are coloured blue. Again $C$ can be given the structure of an $I$-bundle, this time of the form $I\times R_1$, where $R_1$ is one of the two red faces of $C$. To form $D(M_B)\cut DR$, double along the blue bigon faces. This glues the $I$-bundle $C$ to an identical $I$-bundle along blue faces of the form $I\times(\mbox{arc of }\bdy R_1)$, preserving the $I$-bundle structure. Thus $C$ yields a component that is not part of the guts.

To prove the converse statement, suppose that $C$ is an ideal polyhedron obtained from cutting one of the ideal polyhedra making up $D(M_B)\cut DR$ along a square that forms a larger annulus. Suppose $C$ is not part of $\guts(D(M_B)\cut DR)$. Then $C$ belongs to an $I$-bundle or Seifert fibered piece of $D(M_B)\cut DR$. Suppose $C$ is a sub-$I$-bundle of a larger $I$-bundle $Y$. The vertical boundary components of the $I$-bundle $Y$ are vertical annuli with both boundary components on $\widetilde{DR}$ in $D(M_B)\cut DR$, and similarly for $C$. Because the annuli decompose into squares by \reflem{AnnulusSquares}, the fibres must run parallel to an arc of the blue faces. Because the blue faces glue up to form an essential surface, parallel to one fiber, the blue faces must be a vertical surface, and so each blue face of each annulus is fibred. It follows that each blue face must be of the form $I\times\mbox{arc}$, and therefore each blue face is a bigon. Now $C$ is a sub-$I$-bundle of $Y$, homeomorphic to a ball with boundary made up of red faces and blue bigons. It follows that $C$ has the combinatorics of a $(2,q)$-torus link made up of a string of blue bigons.

Finally, suppose $C$ belongs to a Seifert fibered piece $Z$ of $D(M_B)\cut DR$. The Seifert fibering induces a fibering of $\widetilde{DR} \cap Z$. The boundary $\bdy Z$ is a torus, and the vertical boundary of $C$ is an annulus that is broken into squares when put into normal form. It follows that in this case, the red faces of $\widetilde{DR}\cap Z$ are fibered, and hence they form bigons. Then, again, $C$ has the combinatorics of a $(2,q)$-torus link, this time with red bigons. 
\end{proof}

We still wish to identify exactly the guts from a diagram. It becomes slightly easier if we consider the manifold
\[M_{BR} := D(D((S^3-L)\cut B)\cut DR) \]
of \refprop{GutsHomeomorphisms}. 

\begin{lemma}\label{Lem:LargerManifold}
Let $L$ be a link with a reduced, prime, alternating diagram, and associated checkerboard surfaces $B$ and $R$. 
The manifold $M_{BR} = D(D((S^3-L)\cut B)\cut DR)$ has the following properties:
\begin{enumerate}
\item\label{Itm:JSJ} Any maximal disjoint collection of nontrivial squares in the diagram induces a torus decomposition of $M_{BR}$, which contains the tori of the JSJ decomposition. 
\item\label{Itm:NewPolyhedra} Each component of the torus decomposition is built of ideal polyhedra obtained from the original checkerboard polyhedra by cutting along squares.
\item\label{Itm:SF-2qTorus} The Seifert fibered pieces of the decomposition are exactly those obtained by gluing polyhedra with the combinatorics of a $(2,q)$-torus link. All other polyhedra give hyperbolic pieces.
\end{enumerate}
\end{lemma}

\begin{proof}
Note that $M_{BR}$ is obtained by starting with the checkerboard polyhedra of $S^3-L$, doubling along blue faces, then doubling along red faces. Thus it has a decomposition into ideal polyhedra with a checkerboard coloring.

For \refitm{JSJ}, note first that any essential torus can be put into normal form with respect to the polyhedral decomposition. By an Euler characteristic argument, each normal disk making up the torus must be a square. Because the polyhedra have the same combinatorics as the diagram of the alternating link, any collection of essential tori gives a collection of squares.

On the other hand, a nontrivial square in the diagram determines a nontrivial square in the checkerboard polyhedra. When we double across blue and then red faces, the square becomes a torus in $M_{BR}$. We may cut the checkerboard polyhedra along these squares and collapse the squares to ideal vertices, obtaining new ideal polyhedra. If the torus is compressible, a compressing disk can be put into normal form with respect to these new polyhedra. But as in the proof of \reflem{EssentialAnnuliStep2}, this leads to a contradiction if the square is nontrivial. Also as in the proof of \reflem{EssentialAnnuliStep2}, the torus is boundary parallel if and only if the square cuts off a single ideal vertex. Thus a maximal collection of disjoint nontrivial squares in the diagram bounding more than one crossing on each side gives a maximal collection of embedded essential tori in $M_{BR}$. By the uniqueness of the JSJ decomposition, this contains the tori of the JSJ decomposition.

Item~\eqref{Itm:NewPolyhedra} now follows immediately from \refitm{JSJ}. Cut along the squares and collapse each square to an ideal vertex to obtain the new polyhedral decomposition.

For \refitm{SF-2qTorus}, note first that if a polyhedron has the combinatorics of a $(2,q)$-torus link, then it has an $I$-bundle structure of the form $F_1\times I$, where $F_1$ is one of the two faces that is not a bigon. When we double and then double, the $I$-bundle becomes an $S^1$-bundle, which is doubled across annuli on its boundary in a way that preserves the fiber. Hence it is Seifert fibered.

Now suppose a polyhedron $C$ in the complement of the torus decomposition of $M_{BR}$ glues to give a Seifert fibered component $S$. Note first that each ideal vertex of $C$ glues under the doubling to a torus in the parabolic locus, hence the Seifert fibered component $S$ has infinite fundamental group. Note next that the component admits two involutions: reflection through its intersection with the blue surface and reflection through its intersection with the red surface. It follows from work of Meeks and Scott that the involutions preserve the Seifert fibering \cite{MeeksScott}.

Consider the intersections of the blue and red surfaces with $S$. These are incompressible surfaces, hence either vertical or horizontal in $S$. Suppose first that a component of the blue surface is vertical. Then it is parallel to the fibers of the Seifert fibering. Its intersection with $C$ must have corresponding blue faces also parallel to the fibers. Because all faces of $C$ are disks, it follows that these blue faces are of the form $(\mbox{arc})\times I$, or bigon faces. The red meeting the blue cannot also be bigons, else $C$ is formed of four bigons and two ideal vertices, contradicting the fact that the squares in the decomposition were chosen to bound at least two ideal vertices on each side. So the red face must be horizontal in this case. Reflection in the red surface preserves $M_{BR}$, taking blue surfaces to blue. Hence $C$ has the combinatorics of a $(2,q)$-torus link with blue bigon faces.

Suppose instead that each component of the blue surface meeting $C$ is horizontal. Note that the torus boundary components of $S$ are fibered. It follows that the (truncated) ideal vertices of $C$ are fibered of the form $\beta\times I$, where $\beta$ is an arc in the blue face. Then $\bdy\beta\times I$ lies in a red face of $C$. It follows that the red surface must be vertical. Then an argument identical to that above implies that red faces in $C$ are bigons, and again $C$ has the combinatorics of a $(2,q)$-torus link. 
\end{proof}

\begin{lemma}\label{Lem:GutsHomeomorphisms}
  The manifold $D(\guts(D(M_B)\cut DR))$ is homeomorphic to:
  \begin{itemize}
  \item $\guts(M_{BR})$, i.e.\ the hyperbolic part of $M_{BR}$,
  \item $\guts(M_{RB}) = \guts \, D(D((S^3-L)\cut R)\cut DB)$, and
  \item $D(\guts(D(M_R)\cut DB))$.
  \end{itemize}
\end{lemma}

\begin{proof}
We show first that $\guts(M_{BR})$, the manifold consisting of the hyperbolic components of $M_{BR}$ under the JSJ decomposition, is homeomorphic to $\guts(M_{RB})$. This is straightforward: the manifolds $M_{BR}$ and $M_{RB}$ are both obtained from eight copies of one of the checkerboard polyhedra of the link $L$, by doubling along red and blue faces, although in different orders. We obtain a homeomorphism $M_{BR}\cong M_{RB}$ by taking a homeomorphism of polyhedra, and then gluing by the identity across faces. Then $\guts(M_{BR})\cong \guts(M_{RB})$ by the uniqueness of the JSJ decomposition.

Next we show $D(\guts (D(M_B)\cut DR))$ is homeomorphic to $\guts(M_{BR})$. The argument that $D(\guts(D(M_R)\cut DB))$ is homeomorphic to $\guts(M_{RB})$ is symmetric. By \reflem{LargerManifold}, we obtain $\guts(M_{BR})$ by taking a maximal collection of disjointly embedded nontrivial squares, which bound at least two crossings on both sides, cutting along them, and discarding components with the combinatorics of a $(2,q)$-torus link. Then double remaining polyhedra across their red and blue faces. By \reflem{EssentialAnnuliStep1} and \reflem{EssentialAnnuliStep2}, we obtain $D(\guts(D(M_B)\cut DR))$ by first, cutting along a collection of squares corresponding to red bigons and cycles of fused units, then cutting along remaining nontrivial squares that are not parallel to one of these squares, nor parallel to a single ideal vertex, and again discarding components with the combinatorics of a $(2,q)$-torus link by \reflem{IBundlesBigonStrings}.

If there are no fused units in the diagram, then a maximal collection of squares used to create $D(\guts (D(M_B)\cut DR))$ is a maximal collection of squares used to create $\guts(M_{BR})$, by Lemmas~\ref{Lem:EssentialAnnuliStep1} and~\ref{Lem:EssentialAnnuliStep2}. To build the guts in both cases, we decompose along these squares and throw away components with the combinatorics of a $(2,q)$-torus link, by \reflem{IBundlesBigonStrings} and \reflem{LargerManifold}. Then obtain the manifolds by doubling along blue and red faces. These are homeomorphic. 

In the case that there is a cycle of fused units, then note in the first step of the decomposition of $M_B$ we cut along distinct collections of squares in the two checkerboard polyhedra, and these squares intersect each other; see \reffig{FusedUnit}. Hence we must choose one collection of squares to complete to a maximal collection to form $M_{BR}$. However, in both polyhedra the exterior of the cycle of fused units has the combinatorics of a $(2,q)$-torus link, so is discarded. And in both, in the second step of the decomposition there will be two squares within each fused unit that decompose the fused unit into two ``units'' and a $(2,3)$-torus link. Thus only the units remain in the decomposition of $\guts( D(M_B)\cut DR)$. Similarly, in the decomposition of $\guts(M_{BR})$, choose squares from the fused units and build a maximal collection of disjoint squares. After discarding $(2,q)$-torus links, at most the units remain in the guts decomposition. Again there is a homeomorphism of polyhedra, and these are identified in both manifolds by doubling along blue and red faces, so the manifolds are homeomorphic. 
\end{proof}

\begin{proof}[Proof of \refprop{GutsHomeomorphisms}]
  The homeomorphism result of the proposition is by \reflem{GutsHomeomorphisms}. The fact that the manifolds are obtained by eight copies of polyhedra as claimed follows from \reflem{GutsPolyhedra}.
\end{proof}

\begin{proof}[Proof of \refthm{RightAngledGuts}]
In the hyperbolic structure on $D(\guts(D(M_B)\cut DR))$, the red surface is preserved by a reflection, thus as a consequence of Mostow--Prasad rigidity, it must be totally geodesic. Similarly, the blue surface in $D(\guts(D(M_R)\cut DB))$ is totally geodesic. By \reflem{GutsHomeomorphisms}, these manifolds are homeomorphic, hence isometric again by Mostow--Prasad rigidity. It follows that red and blue surfaces are totally geodesic in both. 

Finally, reflection in the red fixes the red surface pointwise, and takes the blue surface to a totally geodesic surface intersecting the red. Similarly, reflection in the blue fixes the blue pointwise and takes the red to a totally geodesic surface. This is possible only if the two surfaces meet at right angles. 
\end{proof}

\section{Topology}\label{Sec:topology}

The next part of the story is topological. In the previous section, we identified guts by considering squares in the diagram graph. Topologically, cutting along a square can be seen as pulling a tangle out of the diagram. We review here some of the literature on tangle decompositions of alternating links. We will see that the decomposition into guts polyhedra above is related to decompositions into algebraic parts of an alternating link, due in part to Conway \cite{Conway}, Bonahon and Siebenmann \cite{BonahonSiebenmann}, and especially Thistlethwaite \cite{Thistlethwaite:Algebraic}. However, there are some subtle differences. In this section, we review many of the results discovered by Thistlethwaite, and use them to build ideal polyhedra associated with an alternating link.

Following Thistlethwaite, we define a \emph{tangle} to be a pair $(X,T)$ where $X$ is homeomorphic to $S^3$ with a finite number of balls removed, and the set $T\subset X$ is a 1-manifold, properly embedded in $X$, such that for every 2-sphere component $F$ of $\bdy X$, $\bdy T\cap F$ consists of four points. Note that this generalises the typical definition of a 2-tangle, in which $X$ is a ball and $T$ is a 1-manifold inside the ball meeting $\bdy X$ in four points. Figure~\ref{Fig:ThistlethwaiteElementary} shows three different examples of tangles using this more general definition. 

\begin{figure}
  \includegraphics{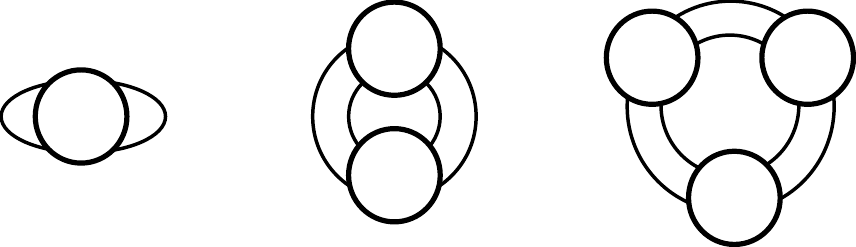}
  \caption{Left to right: A trivial tangle with one boundary component, a trivial tangle with two boundary components, a hollow elementary tangle.}
  \label{Fig:ThistlethwaiteElementary}
\end{figure}

A tangle $(X,T)$ is \emph{trivial} either if $X$ has one boundary component (is a ball) and $(X,T)$ is homeomorphic by a homeomorphisim of pairs to the tangle $(B^3, \mbox{two unknotted arcs})$, or if $X$ has two boundary components and $(X,T)$ is homeomorphic by a homeomorphism of pairs to the tangle $(S^2\times I, \mbox{four unknotted arcs})$ where the endpoints of each arc lie on different boundary components of $S^2\times I$. Note that a homeomorphism of a tangle with one boundary component to a trivial tangle is well-known to be determined by a rational number \cite{Conway}, and thus such a tangle is also called a \emph{rational} tangle. 

The \emph{hollow elementary tangle} is homeomorphic to the tangle $(X,T)$ where $X$ has three boundary components, and $T$ consists of six unknotted arcs with a pair of arcs between each pair of boundary components of $X$. \reffig{ThistlethwaiteElementary} shows two trivial tangles and the hollow elementary tangle.  An \emph{elementary} tangle is obtained from the hollow elementary tangle by gluing a trivial tangle into zero, one, or two of its boundary components. 

For our purposes in this paper, we define a \emph{Conway sphere} in a tangle $(X,T)$ or link $(S^3,L)$ to be a 2-sphere $F$ in the interior of $X$ meeting $T$ transversely in four points.
Note this differs from Thistlethwaite's definition in \cite{Thistlethwaite:Algebraic}: he requires his Conway spheres to be such that neither component of $(X,T)-(F,F\cap T)$ is a trivial tangle with one or two boundary components, forcing the 4-punctured sphere to be essential in $X-T$. We will refer to Thistlethwaite's spheres as \emph{essential} Conway spheres, and use the term Conway sphere to refer to the much more general situation. In any case, we say two Conway spheres are \emph{parallel} if they co-bound a trivial tangle with two boundary components.

A note on historical definitions that we will compare to ours: 
Thistlethwaite defines a link or tangle to be \emph{algebraic} if it is elementary, or if it can be cut along a collection of essential Conway spheres into elementary tangles.
For a tangle $(X,T)$ (or a link $(S^3,L)$) Bonahon and Siebenmann consider the double cover $\widetilde{X}$ of $X$ branched over $T$ \cite{BonahonSiebenmann}. Any essential Conway sphere lifts to an essential torus, and any algebraic tangle lifts to a graph manifold.

It follows by \cite{menasco84, Thistlethwaite:Algebraic} that
$(S^3,L)$ contains a maximal finite collection of pairwise disjoint
and non-parallel essential Conway spheres $F_1, \dots, F_n$, and this
collection is unique outside of elementary tangles.  The
\emph{algebraic part} of $(S^3,L)$, defined by Thistlethwaite in
\cite{Thistlethwaite:Algebraic}, is the union of the closure of
components of $S^3-\cup F_i$ that are elementary.

In \cite{menasco84}, Menasco showed that within an alternating link
diagram, essential Conway spheres can have one of two forms, visible
or hidden, corresponding to \reffig{VisibleHidden}.

\begin{definition}\label{Def:VisibleConwaySphere}
For a given alternating link diagram, a \emph{visible Conway sphere}
intersects the plane of projection in a simple closed curve meeting
the link diagram transverely in four points. We also require that a
visible Conway sphere bounds at least two crossings on each side.
\end{definition}

In the case of a hidden essential Conway sphere, the diagram of the
link always resembles that of the Borromean rings, with four tangles
added, and the hidden Conway sphere meets the plane of projection in
two curves.  By \cite[Proposition~5.1]{Thistlethwaite:Algebraic}, a
visible essential Conway sphere is visible in any alternating diagram
of the link. Consequently, a hidden essential Conway sphere is hidden
in any alternating diagram of the link; see also
\cite{HassThompsonTsvietkova}.  For purposes of this paper, we may
completely ignore hidden essential Conway spheres. However we will
need to consider visible ones, both essential and inessential.  For
example, by \refdef{VisibleConwaySphere}, the boundary sphere of a
trivial tangle (rational tangle) with at least two crossings in a link
diagram is a visible Conway sphere.

\begin{figure}
  \includegraphics{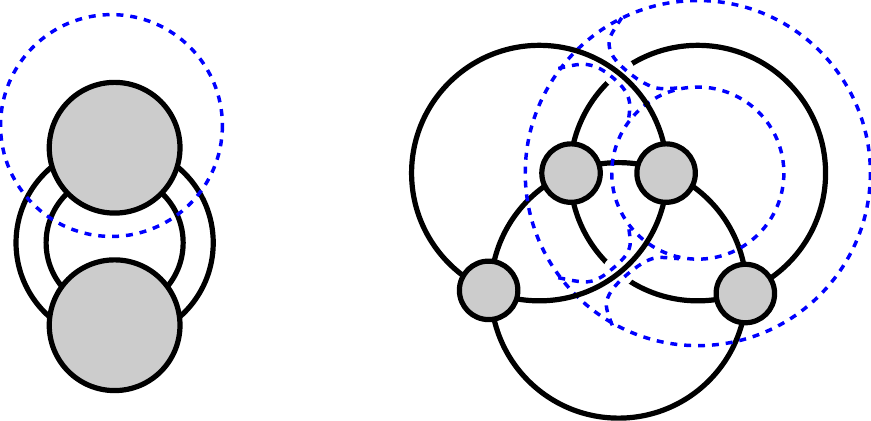}
  \caption{Left: The dashed line shows the intersection of a visible
    Conway sphere with the link diagram. Right: A hidden Conway sphere
    intersects the projection plane in two concentric circles,
    and has two saddles where the two overstrands cross. }
  \label{Fig:VisibleHidden}
\end{figure}

\begin{definition}\label{Def:VisibleAlgebraicPart}
Let $F_1, \dots, F_n$ be a maximal, pairwise disjoint and non-parallel
collection of visible Conway spheres in a reduced, prime, alternating
diagram $D$. The \emph{visible algebraic part} of $D$ is the union of
components of $S^3-\cup F_i$ that are elementary. The diagram $D$ is
\emph{visibly algebraic} if all components of $S^3-\cup F_i$ are
elementary.
\end{definition}

Note that the visible algebraic part differs from Thistlethwaite's algebraic part in two important ways. First, the two differ on diagrams that resemble the Borromean rings. The visible algebraic part will not contain tangles as shown on the right of \reffig{VisibleHidden}, but Thistlethwaite's algebraic part will contain such tangles. Second, the two differ in the presence of a visible \emph{inessential} Conway sphere that bounds a rational tangle that is not a subset of a larger algebraic tangle. Thistlethwaite will completely ignore these rational tangles, but we include them in our visible algebraic part. 

A tangle diagram $(X,T)$ is a regular projection of a tangle onto a plane of projection (a 2-sphere), with over-under crossing information added to the projection of $T$. A crossing $x$ in a tangle diagram is \emph{inessential} if there exists an arc $\alpha$ on the plane of projection with $\bdy \alpha$ lying on the same component of $\bdy X$ and such that $\alpha$ meets $T$ exactly in the point $x$; see \reffig{ThistlethwaiteReduced}, which is adapted from \cite{Thistlethwaite:Algebraic}. Note that we may perform a flype to move the crossing to be adjacent to $\bdy X$, then isotope two points of $\bdy T$ on $\bdy X$ to remove the crossing as in \reffig{ThistlethwaiteReduced}, right. A tangle diagram is \emph{reduced} if it contains no inessential crossings.

\begin{figure}
  \includegraphics{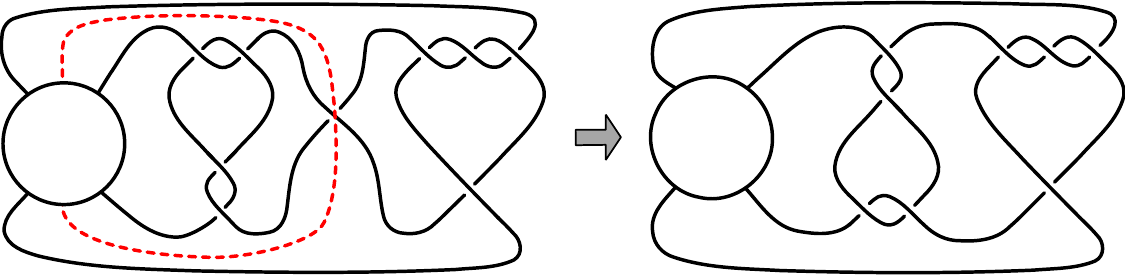}
  \caption{Left: An inessential crossing. Right: The crossing can be removed.}
  \label{Fig:ThistlethwaiteReduced}
\end{figure}

\begin{definition}\label{Def:CrossingClosure}
Let $(X,T)$ be a tangle with a reduced, prime alternating
diagram. Attach to each component of $\bdy X$ the trivial tangle with
one boundary component whose diagram has a single crossing, as in
\reffig{CrossingClosure}.  Since $T$ is an alternating tangle
diagram, the four points in $T\cap\bdy X$ alternate as endpoints of
over-crossing and under-crossing arcs of $T$.  Thus, for each added
crossing as in \reffig{CrossingClosure}, we can choose its sign such
that the diagram of the resulting link is alternating. Such an
alternating link is called the \emph{crossing closure} of the tangle
$(X,T)$.
\end{definition}

\begin{figure}
  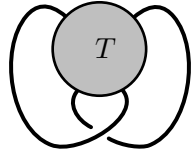
  \caption{The crossing closure of tangle $T$.}
  \label{Fig:CrossingClosure}
\end{figure}

\begin{definition}[Tangle polyhedra]\label{Def:TanglePolyhedra}
Let $L$ be a link with a reduced, prime, alternating diagram $D$. Define the \emph{tangle polyhedra} associated with $L$ as follows:

If $D$ is visibly algebraic, then define the \emph{tangle polyhedra} to be the empty set.

If $D$ admits no visible Conway sphere, then define the \emph{tangle polyhedron} to be the checkerboard polyhedron for $D$.

Otherwise, cut $D$ along a maximal collection of pairwise disjoint,
non-parallel, visible Conway spheres $F_1, \dots, F_n$.  Remove the
visible algebraic part. What remains consists of a nonempty collection
of tangles. For each, reduce the tangle diagram to remove all
inessential crossings. Next, form the crossing closure of each tangle,
giving a new collection of alternating link diagrams that are tangle closures
of non-algebraic tangles. Finally, the \emph{tangle polyhedra} are the
union of the checkerboard polyhedra of the new reduced alternating
link diagrams.
\end{definition}

\begin{example}\label{Ex:TanglePolyhedra}
  We provide here a few examples of tangle polyhedra for different diagrams. 
  \begin{enumerate}
  \item By definition, the tangle polyhedra of any visibly algebraic link is the empty set. This includes 2-bridge knots, Montesinos knots, and more complicated algebraic links. The polyhedra of this paper do not give useful information for these links; see the discussion in \refsec{Discussion}.
  \item The usual diagram of the Borromean rings admits no visible Conway spheres. Thus the tangle polyhedra consist of two ideal octahedra, which agree with the checkerboard polyhedra, as described in Thurston's notes \cite{tnotes}. 
  \item A {\em weaving knot} $W(p,q)$ is the alternating knot or link with the same projection as the standard closed $p$--braid $(\sigma_1\ldots\sigma_{p-1})^q$ diagram of the torus knot or link $T(p,q)$. See Section~\ref{sec:Wpq} for more details on weaving knots. For $p,q\geq 3,\ W(p,q)$ admit no visible Conway spheres, thus again their tangle polyhedra consist of two ideal polyhedra with the same combinatorics as the diagram of the weaving knot.
  \item Take a weaving knot diagram, as above, but replace one crossing by a rational tangle with at least two crossings. By Thistlethwaite's definition of the algebraic part of an alternating link, the resulting link has no algebraic part. However, the \emph{visible} algebraic part of this link, as in \refdef{VisibleAlgebraicPart}, consists of that rational tangle. Removing it and taking the crossing closure of the result gives back the original weaving knot diagram. Thus, the tangle polyhedra of the weaving knot with a crossing replaced by a rational tangle agree with the tangle polyhedra of the original weaving knot.
  \end{enumerate}
\end{example}

In general, replacing any crossing of a diagram by a rational or algebraic tangle does not affect the tangle polyhedra of the result. 

\begin{theorem}\label{Thm:LinkInvariant}
Let $L$ be a link with a reduced, prime, alternating diagram. The tangle polyhedra associated with a given diagram of $L$ are well-defined, and independent of choice of alternating diagram of $L$. Therefore they are a link invariant. 
\end{theorem}

\begin{proof}
By the proof of the Tait flyping conjecture
\cite{menasco-thist:alternating}, any two alternating diagrams of $L$
differ by a finite sequence of flypes. We will show that a single
flype does not affect the tangle polyhedra.  It follows that any
finite sequence of flypes does not affect the tangle polyhedra of $L$.

We consider the cases in \refdef{TanglePolyhedra} for the reduced,
prime, alternating diagram $D$ of $L$.  If $D$ admits no visible
Conway sphere, then $D$ does not admit flypes.  Suppose $D$ admits a
flype along a visible Conway sphere $F$ at a crossing $c$.  We can
assume that $F$ is part of the maximal collection of pairwise
disjoint, non-parallel, visible Conway spheres for $D$, as $F$ can be
chosen first.  If $D$ is visibly algebraic, then by definition it
remains so after the flype.  Before the flype, $c$ is removed in one
of two ways: Either it lies in the visible algebraic part of $D$, or
it is removed as an inessential crossing with respect to the tangle
inside $F$.  After the flype, $c$ is again removed in one of these
ways.  Thus, following \refdef{TanglePolyhedra}, we obtain the same
collection of non-algebraic tangles before and after the flype.

So a flype does not change any of the diagrams of the
alternating links used to construct the tangle polyhedra. Therefore,
the tangle polyhedra form a link invariant.
\end{proof}

\begin{theorem}\label{Thm:TangleEqualsGuts}
Let $L$ be a link with a reduced, prime, alternating diagram. Any
choice of a maximal collection of disjoint squares determines visible
Conway spheres and tangle polyhedra on the one hand, and guts
polyhedra on the other.  Then the associated tangle polyhedra and guts
polyhedra are identical.
\end{theorem}

\begin{proof}
By \refdef{GutsPolyhedra}, the guts polyhedra are obtained by cutting the checkerboard polyhedra associated with the diagram of $L$ along a maximal collection of disjoint squares. The squares that have been cut become ideal vertices. By \reflem{IBundlesBigonStrings}, any component with the combinatorial form of a $(2,q)$-torus link is discarded. The remaining components are the guts polyhedra. 

On the other hand, we claim that the tangle polyhedra are also
obtained by considering a maximal collection of disjoint squares.  For
a prime alternating link, the projection graph of the reduced
alternating diagram and the checkerboard polyhedral graph are the
same.  Hence, by separating the diagram of $L$ into tangles along a
maximal collection of disjoint squares, we get reduced alternating
tangle diagrams that do not admit a visible Conway disk, as defined in
\cite{Thistlethwaite:Algebraic, Thistlethwaite_JKTR}.  By the
classification of alternating tangles in \cite{Thistlethwaite_JKTR},
these are exactly the tangles obtained by separating the diagram of
$L$ along visible Conway spheres.  Hence, separating the diagram of
$L$ along squares as in \refdef{GutsPolyhedra} or along visible
Conway spheres as in \refdef{TanglePolyhedra} results in the same set
of reduced alternating tangle diagrams.

The visible algebraic part consists of the union of tangles
that are either trivial or elementary. Note in
\reffig{ThistlethwaiteElementary} that the elementary tangles have a
diagram with the combinatorics of a $(2,q)$-torus link.  It follows
that we remove exactly the same portion of the diagram to form guts
polyhedra and tangle polyhedra. In the case of the tangle polyhedra,
we take the crossing closure, inserting a crossing into each
square. This causes the associated checkerboard polyhedra to have an
ideal vertex exactly in the location of the square, which is exactly
where the guts polyhedra have an ideal vertex. Thus, the combinatorial
polyhedra are identical.
\end{proof}

\begin{corollary}
  For a prime, alternating link $L$:
  \begin{itemize}
  \item The guts polyhedra for $L$ give a link invariant.
  \item The tangle polyhedra admit a right-angled ideal hyperbolic structure.
  \end{itemize}
\end{corollary}

\section{Combinatorics}\label{Sec:combinatorics}

The third part of the story is the combinatorics of the diagram graph, which has already played a role in establishing the correspondence between guts polyhedra and tangle polyhedra.

Start with a reduced, twist-reduced, prime alternating diagram of a link
$L$. This has a 4-valent projection graph $\Gamma(L)$, which may have
bigons, and a planar dual graph $\Gamma^*(L)$.
A \emph{$k$-circuit} is a simple closed curve composed of $k$ edges of
a graph.  In this paper, we will consider only $4$-circuits on
$\Gamma^*(L)$ arising from a reduced, twist-reduced, prime alternating link
diagram.  We say that a $4$-circuit of $\Gamma^*(L)$ is \emph{trivial}
if it bounds a single crossing of $L$ on either side, and otherwise it
is \emph{nontrivial}.

To avoid ambiguity, we will refer to the diagram of a trivial tangle with either one or two boundary components as a \emph{rational tangle diagram}.

\begin{definition}\label{Def:CrossingParallel}
  Two $4$-circuits are \emph{crossing-parallel} if they differ only by
  passing on opposite sides of a single crossing, as in
  \reffig{Parallel4Circuit}. Two $4$-circuits $A$ and $B$ are
  \emph{parallel} if there is a sequence of $4$-circuits $A_1=A,\,
  A_2,\,\dots,\, A_n=B$, with $A_j$ crossing-parallel to $A_{j+1}$ for
  $j=1, \dots, n-1$.
\end{definition}

\begin{figure}
  \includegraphics{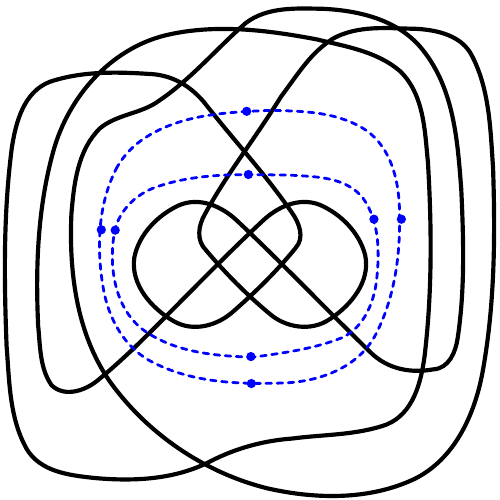}
  \caption{Two crossing-parallel $4$-circuits, shown as dashed lines. Although the $4$-circuits share two edges, we sketch one pushed slighty inside the other.}
  \label{Fig:Parallel4Circuit}
\end{figure}

We will consider a nontrivial $4$-circuit on the projection plane
for $\Gamma(L)$.  Capping off by disks on either side of the
projection plane, the $4$-circuit gives an embedded 4-punctured
sphere, which is a visible Conway sphere.

\begin{lemma}\label{Lem:Parallel}
Let $A$ and $B$ be $4$-circuits in $\Gamma^*(L)$ with corresponding 4-punctured spheres $\overline{A}$ and $\overline{B}$.
Then the following are equivalent:
\begin{enumerate}
\item $A$ and $B$ are parallel,
\item $\overline{A}$ and $\overline{B}$ are ambient isotopic in $S^3-L$,
\item $A$ and $B$ cobound a rational tangle diagram.
\end{enumerate}
\end{lemma}

\begin{proof}
If $A$ and $B$ are parallel, then there exists a sequence of
crossing-parallel $4$-circuits between them. These determine a
sequence of embedded 4-punctured spheres, each pair of which encloses
a single crossing of the diagram. The region enclosed is homeomorphic
to $S_4\times I$, where $S_4$ denotes the 4-punctured sphere. Thus, we
have an ambient isotopy from one side of the crossing to the
other. Putting these together gives the ambient isotopy from
$\overline{A}$ to $\overline{B}$. Hence, $(1)$ implies $(2)$.

If $\overline{A}$ and $\overline{B}$ are ambient isotopic, then they
cobound a trivial tangle with two boundary components.  We can
represent this isotopy by a rational tangle diagram, as pointed out in
the remark after \cite[Corollary~3.2]{Thistlethwaite:Algebraic}.
Thus, $(2)$ implies $(3)$.

By \cite[Corollary~3.2]{Thistlethwaite:Algebraic}, if $A$ and $B$
cobound a rational tangle diagram, then it is either unreduced or has
no crossings.  Thus, the original unreduced alternating diagram is
obtained by adding one crossing at a time, adjacent to the $4$-circuit
$A$ (or $B$), and so $A$ and $B$ are parallel.  Hence, $(3)$ implies
$(1)$.
\end{proof}

If a $4$-circuit gives a visible Conway sphere that bounds a rational
tangle diagram, then that $4$-circuit is parallel to a trivial
$4$-circuit, which bounds one crossing.  However, the trivial
$4$-circuit inside a rational tangle diagram is not uniquely
determined.  Thus, by \reflem{Parallel}, $4$-circuits $A$ and $B$ are
parallel if and only if they cobound a rational tangle diagram, but
such a pair may not be uniquely determined by the tangle diagram.

\begin{definition}\label{Def:BoundingPair}
A pair of parallel $4$-circuits $A$ and $B$ is called a \emph{maximal
  bounding pair} if $A$ and $B$ cobound a rational tangle diagram
$\tau$, and there do not exist parallel $4$-circuits $A'$ and $B'$
that cobound a rational tangle diagram $\tau'$ which contains $\tau$
as a sub-tangle.  Two maximal bounding pairs $\{A,\,B\}$ and
$\{A',\,B'\}$ are \emph{disjoint} if they cobound disjoint rational
tangle diagrams.
\end{definition}

\begin{figure}
  \includegraphics{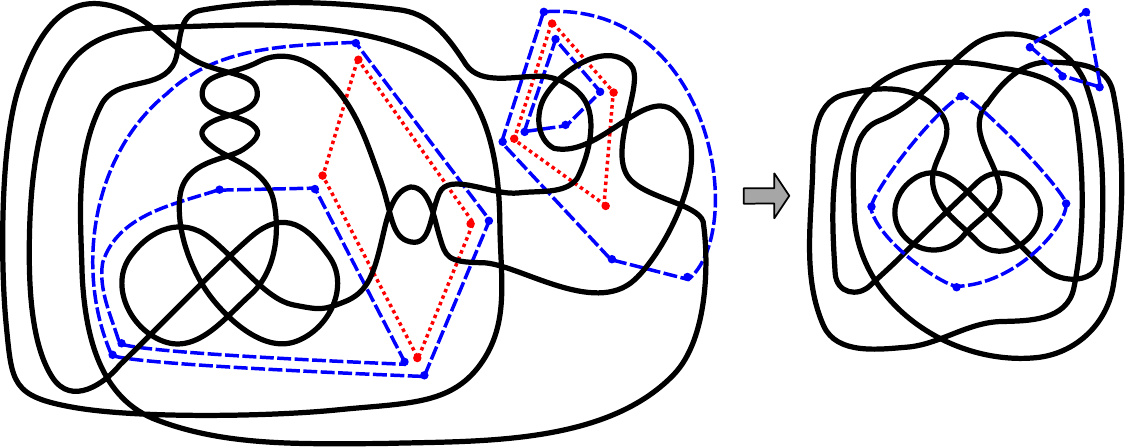}
  \caption{Left: Two disjoint maximal bounding pairs of 4-circuits are shown
(long dashed lines, in blue). Two additional $4$-circuits that are not in a maximal bounding pair are shown (short dotted lines, in red). Right: Rational reduction eliminates all crossings between both maximal bounding pairs of $4$-circuits.}
  \label{Fig:Outermost}
\end{figure}

\begin{definition}\label{Def:RationalReduction}
Let $L$ be a link with a reduced, twist-reduced, prime alternating
diagram.  For all pairwise disjoint maximal bounding pairs of
$4$-circuits in $L$, remove all crossings between each pair by a corresponding homeomorphism of each rational tangle.  Thus,
each rational tangle diagram with one boundary component is replaced
by a single crossing, and all crossings are removed in every rational
tangle diagram with two boundary components.  In the resulting
diagram, some $4$-circuits that were not parallel before may now be
parallel. In that case, repeat the process, removing all crossings
between pairwise disjoint maximal bounding pairs of $4$-circuits.
Because each move reduces the number of crossings, the process
eventually terminates.  We call this \emph{rational reduction} of the
diagram $L$. The final diagram is \emph{rationally reduced}.
\end{definition}

In \reffig{Outermost} left, two disjoint maximal bounding pairs of
$4$-circuits are shown in blue, and two $4$-circuits that are not in a
maximal bounding pair are shown in red.
We use these two maximal bounding pairs to obtain a rationally
reduced diagram in \reffig{Outermost} right.

\begin{definition}\label{Def:PrismaticCircuit}
A \emph{prismatic $4$-circuit} is a $4$-circuit $\gamma$ so that no two edges of $\Gamma(K)$ that meet $\gamma$ share a vertex in $\Gamma(K)$.
\end{definition}

It follows from this definition that a prismatic 4-circuit is
nontrivial. Conversely, we have the following.

\begin{lemma}\label{Lem:4circuits}
Each nontrivial $4$-circuit of a rationally reduced diagram is a prismatic $4$-circuit. 
\end{lemma}

\begin{proof}
If not, the $4$-circuit meets edges that share a vertex, so the
$4$-circuit is adjacent to a crossing. But then the diagram is not
rationally reduced because there exists a pair of crossing-parallel
$4$-circuits.
\end{proof}

A combinatorial polyhedron $P$ is a cell complex on $S^2$ that can be
realized as a $3$--dimensional convex polyhedron.  Steinitz proved
that a graph can be realized as the $1$--skeleton of such a convex
polyhedron if and only if the graph is a $3$--connected simple planar
graph~\cite{steinitz}.  A combinatorial polyhedron is realizable as a
right-angled hyperbolic polyhedron if there exists an ideal hyperbolic
polyhedron with the same combinatorial structure as $P$ and with all
dihedral angles $\pi/2$.

\begin{theorem}[Andreev's theorem for 4-valent right-angled ideal polyhedra]\label{Thm:Andreev}
  A 4-valent combinatorial polyhedron 
  admits a realization as a right-angled ideal hyperbolic polyhedron
  if and only if it has no nontrivial $4$-circuits. The realization
  is unique up to isometry of $\HH^3$.
\end{theorem}

\begin{proof}
This special case follows almost immediately from the version of Andreev's theorem given by Atkinson~\cite[Theorem~2.1]{atkinson:volume}.  
Let $P$ be a 4-valent combinatorial polyhedron and let $\Gamma$ denote its 1-skeleton. We step through the necessary and sufficient conditions of that theorem:
\begin{enumerate}
\item $P$ has at least six faces.
\item Every vertex has degree 3 or 4.
\item For any triple of faces of $P$, $(F_i,F_j,F_k)$ such that $F_i\cap F_j$ and $F_j\cap F_k$ are edges of $P$ with distinct endpoints, $F_i\cap F_k = \emptyset$.
\item $P$ has no prismatic $4$-circuits.
\end{enumerate}
Since $P$ is a combinatorial polyhedron, Steinitz's theorem implies
that $\Gamma$ is a 3-connected simple planar graph. Since $\Gamma$
is also 4-valent, it follows from the census of knots and links by
crossing number (see e.g.~\cite{Rolfsen}) that $\Gamma$ has at least 6
vertices.
An Euler characteristic argument implies $P$ must have at least $8$
faces, hence conditions (1) and (2) always hold. 

Now, if $P$ has no nontrivial 4-circuits then (4) holds by \reflem{4circuits}.
We only need to show that (3) is always satisfied.

\begin{figure}
  \includegraphics{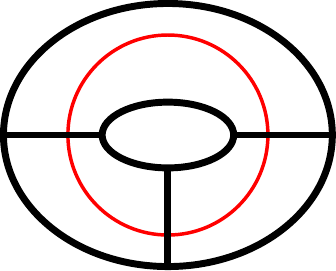}
  \caption{The case that $F_i$ and $F_k$ intersect in an edge.}
  \label{Fig:CaseEdge}
\end{figure}

Let $(F_i,F_j,F_k)$ be a triple of faces such that $F_i\cap F_j$ and $F_j\cap F_k$ are edges of $P$ with disjoint endpoints, and suppose by way of contradiction that $F_i\cap F_k$ is nonempty. 

{\bf Case 1:} $F_i\cap F_k$ contains an edge. Then there exists a
simple closed curve $C$ that intersects exactly these three edges of
$P$, as shown in \reffig{CaseEdge}. 
Let $G$ be the portion of the graph of $P$ in a disk 
bounded by $C$, and let $G'=G\cup C$ be a new graph obtained by adding
three edges and three vertices lying on $C$. Then $G'$ has three
vertices of degree $3$ and all other vertices of degree $4$. This
implies that the sum of all degrees of vertices of $G'$ is an odd
number. But the sum of all degrees of vertices equals twice the number
of edges, which is even. A contradiction. 

{\bf Case 2:} $F_i\cap F_k$ contains a vertex. In this case we can
construct a nontrivial 4-circuit taking the dual edges near the
crossing, contradicting the assumption.

The contradictions in both cases imply that $P$ satisfies condition (3) above, and the result follows from \cite[Theorem~2.1]{atkinson:volume}.

Conversely, suppose $P$ admits a realization as a right-angled ideal
hyperbolic polyhedron, and hence satisfies the four conditions above.
If $P$ has a nontrivial 4-circuit which is not prismatic, then 
using an argument similar to that in \reflem{4circuits},
this would contradict condition (3). Hence, all nontrivial 4-circuits in $P$ 
are prismatic, so then condition (4) implies that there 
are no nontrivial 4-circuits. 
\end{proof}

\begin{lemma}\label{Lem:Andreev}
  Let $\Gamma$ be the 4-valent planar projection graph of a reduced,
  twist-reduced, prime alternating link diagram $K$.  Then $\Gamma$
  admits a realization as a right-angled ideal hyperbolic polyhedron
  if and only if it has no nontrivial $4$-circuits. The
  realization is unique up to isometry of $\HH^3$.
\end{lemma}

\begin{proof}
If $\Gamma$ is a polyhedral graph, then \refthm{Andreev} implies it is realized as a right-angled ideal hyperbolic polyhedron if and only if it has no nontrivial 4-circuits. It remains to show that if $\Gamma$ is not a polyhedral graph, then it has a nontrivial 4-circuit. 

By Steinitz's theorem, $\Gamma$ is a polyhedral graph if and only if
it is a 3-connected simple planar graph.
Suppose $\Gamma$ is not simple.
The projection graph of a
reduced, twist-reduced, prime alternating link diagram is simple if
and only if it has no bigons. If $\Gamma$ has a bigon, then
$\Gamma$ has a nontrivial 4-circuit encircling the bigon.

Suppose $\Gamma$ is not 3-connected. Then a curve running through the two vertices that disconnect $\Gamma$ can be pushed slightly off those vertices to give a nontrivial 4-circuit, using the fact that the diagram is prime.
\end{proof}

\begin{definition}[cf\ Atkinson~\cite{atkinson:polyhedra}]\label{Def:Split}
Let $P$ be 4-valent planar graph with no bigon regions. If $\gamma$ is a prismatic $4$-circuit for the dual graph $P^*$, we define $P$ \emph{split along} $\gamma$, denoted $P\split \gamma$, as follows: Choose a planar embedding of $P$. Form two new graphs $P_{\text{int}}$ and $P_{\text{ext}}$, where $P_{\text{int}}$ (respectively, $P_{\text{ext}}$) consists of all edges and vertices of $P$ in the bounded (respectively, unbounded) component of $\R^2-\gamma$, such that $P_{\text{int}}$ and $P_{\text{ext}}$ each have four $1$-valent vertices which were incident to $\gamma$. Let $\P_{\text{int}}$ (respectively, $\P_{\text{ext}}$) be the $4$-valent graph obtained by taking the edges from each of the 1-valent vertices, and attaching them to a single vertex chosen to lie in the unbounded (respectively, bounded) region of $\R^2-\gamma$.  Then $P \split \gamma$ consists of the disjoint union of $\P_{\text{int}}$ and $\P_{\text{ext}}$. See \reffig{Prismatic}.
\end{definition}

\begin{figure}
  \begin{center}
    \includegraphics{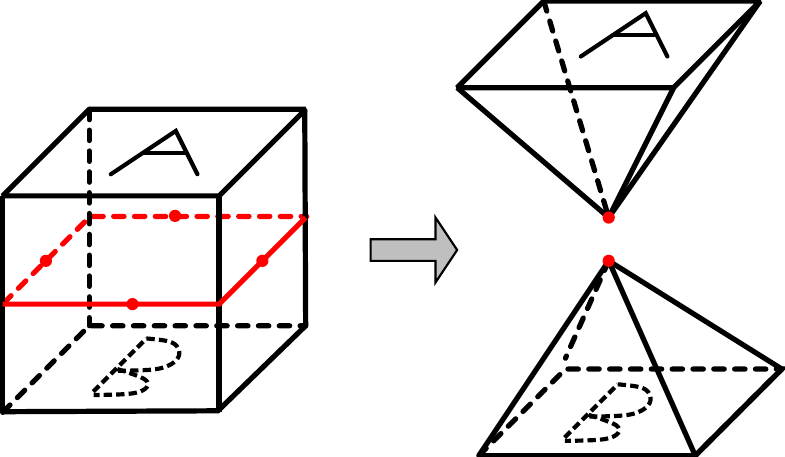}
\caption{A polyhedron $P$ is split along a prismatic $4$-circuit
  $\gamma$ (whose dual $4$-circuit is shown in the center, in red) to obtain $P\split \gamma$, as in \refdef{Split}.}
\label{Fig:Prismatic}
  \end{center}
\end{figure}

\begin{definition}[Andreev polyhedra]\label{Def:AndreevPolyhedra}
Start with a reduced, twist-reduced, prime, alternating link diagram $L$.  Let $\Gamma(L)$ be the projection graph of its rationally reduced diagram.  Split $\Gamma(L)$ iteratively along nontrivial $4$-circuits into graphs that either
\begin{enumerate}
\item [$(i)$] have exactly one vertex,
\item [$(ii)$] have nontrivial $4$-circuits, or
\item [$(iii)$] have no nontrivial $4$-circuits.
\end{enumerate}
We discard all graphs in case $(i)$.  For graphs in case $(ii)$, we repeat this process as needed: Rationally reduce the corresponding link diagram, and then split along nontrivial $4$-circuits as above.  Because each move reduces the number of vertices, the process eventually terminates.  Finally, the only remaining graphs have no nontrivial $4$-circuits.  By \reflem{Andreev}, each such graph admits the structure of a right-angled ideal hyperbolic polyhedron.  The resulting set of right-angled ideal hyperbolic polyhedra are called the \emph{Andreev polyhedra} associated to $L$.
\end{definition}

\begin{theorem}\label{Thm:AndreevEqualsGuts}
  The Andreev polyhedra are identical to the tangle polyhedra and the guts polyhedra. 
\end{theorem}

\begin{proof}
Let $D$ be a reduced, twist-reduced, prime, alternating link diagram.
We will prove that Definitions~\ref{Def:TanglePolyhedra}
and~\ref{Def:AndreevPolyhedra} agree for $D$.  Rational reduction, as
in \refdef{RationalReduction}, repeatedly replaces all rational tangle
diagrams with one boundary component with a single crossing, and
removes all crossings in a rational tangle diagram with two boundary
components.  A visibly algebraic link diagram $D$ is rationally
reduced to a $1$-crossing diagram, which is discarded in both cases.

We proceed by induction on nontrivial $4$-circuits in $\Gamma(D)$.
By \reflem{Parallel}, every set of parallel nontrivial $4$-circuits
corresponds to a visible Conway sphere.  Hence, $D$ admits no visible
Conway sphere if and only if $\Gamma(D)$ has no nontrivial
$4$-circuits.  In this case, the two checkerboard
polyhedra for $D$ are both its Andreev polyhedra and its tangle
polyhedra.

Proceeding inductively, for Andreev polyhedra, we split $\Gamma(D)$
along a nontrivial $4$-circuit, and obtain either right-angled
polyhedral graphs or graphs with fewer nontrivial $4$-circuits that
are then rationally reduced.  In the latter case, the equivalence with
tangle polyhedra follows by the induction hypothesis.  In the former
case, we claim that these right-angled polyhedra are also tangle
polyhedra for $D$.

To obtain tangle polyhedra, we cut $D$ along visible Conway spheres,
remove the visible algebraic part and inessential crossings, use
crossing closures to form reduced alternating diagrams, and then take
their checkerboard polyhedra.  The crossing closure makes each
checkerboard polyhedron have an ideal vertex exactly where we cut
along the visible Conway sphere, which is the same as splitting
$\Gamma$ along the corresponding nontrivial $4$-circuit.  The only
difference is when rational reduction occurs: for tangle polyhedra,
the tangles are reduced before taking their crossing closures; for
Andreev polyhedra, we split and rationally reduce each alternating
link diagram repeatedly, as needed, discarding rationally reduced
unknots.  In both cases, what remains are reduced alternating diagrams
obtained by taking the crossing closure of each non-algebraic tangle
cut along nontrivial $4$-circuits.  Thus, the resulting combinatorial
polyhedra are identical.

By \refthm{TangleEqualsGuts}, these are also the same as the guts polyhedra.
\end{proof}

\section{Right-angled volume}\label{Sec:volume}

In this section, we use the guts/tangle/Andreev polyhedra to define a
new geometric link invariant, which gives a lower bound on the volume
of the link.

\begin{definition}\label{Def:VolP}
Let $L$ be a link with a reduced, twist-reduced, prime alternating diagram. The \emph{right-angled volume} $\volp(L)$ is defined to be twice the sum of the volumes of the right-angled guts polyhedra, or equivalently by \refthm{AndreevEqualsGuts}, twice the sum of the volumes of the tangle polyhedra, or twice the sum of the volumes of the Andreev polyhedra.  If the set of such polyhedra is empty, we define $\volp(L)=0$.
\end{definition}

For example, if $L$ denotes the standard alternating diagram of the Borromean rings, the guts polyhedra consists of a single regular ideal octahedron. Therefore $\volp(L) = 2\voct$, where $\voct\approx 3.66386$ is the volume of the regular ideal octahedron.

On the other hand, if $L$ is any visibly algebraic link, as in  Example~\ref{Ex:TanglePolyhedra}\,(1), $\volp(L)=0$. 

The three equivalent definitions of $\volp(L)$ imply different properties:  Using tangle polyhedra we prove the invariance of $\volp(L)$ (\refthm{volp-invariant} below); using guts
polyhedra we prove the lower bound for the volume of $S^3-L$ (\refthm{GutsVolumeBound} below); and using Andreev polyhedra we provide a diagrammatic expression for $\volp(L)$ (\refthm{volp-tangles} below).

\begin{theorem}\label{Thm:volp-invariant}
If $L$ is any prime alternating link, then $\volp(L)$ is a link invariant.
\end{theorem}
\begin{proof}
  Using the tangle polyhedra of $L$ to obtain $\volp(L)$, invariance of the polyhedra follows from \refthm{LinkInvariant}.
  The uniqueness of their volume follows from \refthm{Andreev}.
\end{proof}

\begin{theorem}\label{Thm:GutsVolumeBound}
For any hyperbolic alternating link $L$ with hyperbolic volume $\vol(L)$,
\[ \volp(L) \:\leq\: \vol(L). \]
\end{theorem}

\begin{proof}
Here, we consider $\volp(L)$ in terms of the guts polyhedra of $L$.
For a reduced, prime alternating diagram of $L$, the guts polyhedra inherit a hyperbolic structure with geodesic faces meeting at right angles, by \refthm{RightAngledGuts}.

A theorem of Agol, Storm, and Thurston \cite[Theorem~9.1]{AgolStormThurston}, states that for a hyperbolic 3-manifold $N$ with embedded $\pi_1$-injective surface $\Sigma$,
\[ \vol(N) \: \geq \: \half \vtet || D(N\cut\Sigma)|| \: = \: \half\vol(D(\guts(N\cut\Sigma))). \] 
Here, $\vtet$ is the volume of a regular ideal tetrahedron; $||\cdot||$ denotes Gromov norm; \cite[Theorem~9.1]{AgolStormThurston} gives the inequality; and the equality is the definition of the Gromov norm. 

We apply this result twice, to surfaces $B$ and $DR$, which implies
\[ \vol(L) \geq \half \vol(D(\guts((S^3-L)\cut B))) = \half\vol(D(M_B)) \geq \frac{1}{4}\vol(D(\guts(D(M_B)\cut DR))). \]
The manifold $D(\guts(D(M_B)\cut DR))$ is built by gluing eight copies of the guts polyhedra, glued by the identity along red and blue faces. Thus this gives the desired result. 
\end{proof}

\begin{theorem}\label{Thm:volp-tangles}
Let $L$ be a prime alternating link, given by a reduced, twist-reduced, prime alternating diagram.
Let $K$ be its rationally reduced diagram, as in \refdef{RationalReduction}.
\begin{enumerate}  
\item If $K$ has five or fewer crossings, then $\volp(K)=0$.
\item If $K$ admits no nontrivial $4$-circuits, then $\volp(K)=2\,\vol(P(K))$, where $P(K)$ is the checkerboard polyhedron for $K$, with a right-angled ideal hyperbolic structure.
\item Otherwise, split $K$ along nontrivial $4$-circuits to obtain a set of alternating tangles whose crossing closures form alternating link diagrams $K_i$.  Repeatedly, as needed, rationally reduce each $K_i$ and apply steps $(1)-(3)$.  Let $\{T_i\}$ be the resulting set of reduced non-algebraic tangles, and let $T_i^{\times}$ denote the crossing closure of $T_i$, as in \reffig{CrossingClosure}.  Then
  \[ \volp(L) = \sum\nolimits_i \volp(T_i^{\times}). \]
\end{enumerate}
\end{theorem}

\begin{proof}
For (1), any alternating diagram with five or fewer crossings is visibly algebraic.

For (2), the claim follows by \reflem{Andreev}.  In this case, the two
checkerboard ideal polyhedra $P(K)$ are exactly the tangle polyhedra
of $K$.

For (3), we follow the procedure in \refdef{AndreevPolyhedra}.  By the
proof of \refthm{AndreevEqualsGuts}, the $T_i$ are the tangles whose
crossing closures form the tangle polyhedra.  The inductive proof
gives a way to find the tangle diagrams starting from the link diagram
$L$ by repeated rational reduction and splitting.
\end{proof}  

\begin{figure}
  \begin{center}
    \includegraphics{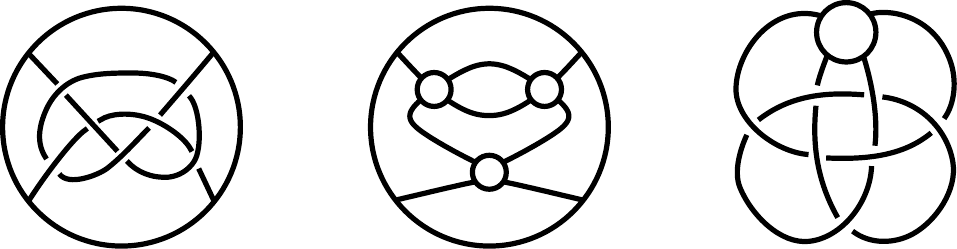}
\caption{The tangles $\tau_1,\ \tau_2,\ \tau_3$ used in Example~\ref{Ex:Tangle-insert}.}
\label{Fig:Tangle-insert-fig}
  \end{center}
\end{figure}

\begin{example}\label{Ex:Borromean}
For the Borromean link $L$, its reduced alternating diagram satisfies condition (2) of \refthm{volp-tangles}.  Hence, $\volp(L)=\vol(L)$.
\end{example}

\begin{example}\label{Ex:Tangle-insert}
Let $\tau_1,\ \tau_2,\ \tau_3$ be the tangles shown in
\reffig{Tangle-insert-fig}.  Let $K$ be the alternating link obtained
by inserting $\tau_1$ into each of the three inner boundary components
of $\tau_2$, and then inserting the resulting tangle into the boundary
of $\tau_3$.  Note $K$ is rationally reduced.

Using \refthm{volp-tangles}, we compute $\volp(K)$.  Following step~(3), we repeatedly split $K$ into alternating diagrams.  For both
$\tau_1$ and $\tau_3$, the crossing closure is the knot $8_{18}$, and
the repeated crossing closure of $\tau_2$ is the figure-eight knot.
Thus we obtain alternating links $K_i$, for \mbox{$1\leq i\leq 5$}, such that four of the $K_i$ are diagrams of $8_{18}$, and the other one is the figure-eight knot diagram.  The figure-eight knot is visibly algebraic, and can be rationally reduced to an unknot.
Hence we obtain four crossing closures of tangles $T_i^{\times}$ in step~(3) of \refthm{volp-tangles}, and each of these four is the knot $8_{18}$. 
Thus $\volp(K)=4\,\volp(8_{18})$. We compute $\volp(8_{18})$ exactly in
Example~\ref{Ex:8_18} below.
\end{example}

\subsection{Computing $\volp(L)$}
We now provide an explicit algorithm to compute $\volp(L)$ from a
reduced, prime, alternating diagram of $L$.  Applying
\refthm{volp-tangles}, we get a set of reduced non-algebraic tangles
$T_i$. We can then apply \refthm{RightAngledKites}, below, to
explicitly compute each $\volp(T_i^{\times})$, and hence $\volp(L)$.

We will compute $\volp(T_i^{\times})$ by dividing right-angled polyhedra into well-understood pieces. \reffig{Orthoscheme} shows one such piece, which is a 3/4-ideal tetrahedron with one vertex at $\infty$, the other two ideal vertices on the boundary of the same hemisphere on $\CC\subset\bdy\HH^3$, and the finite vertex at the point with maximum Euclidean height on that hemisphere. This is the double of what Schl\"afli called an \emph{orthoscheme}.

\begin{figure}
  \includegraphics{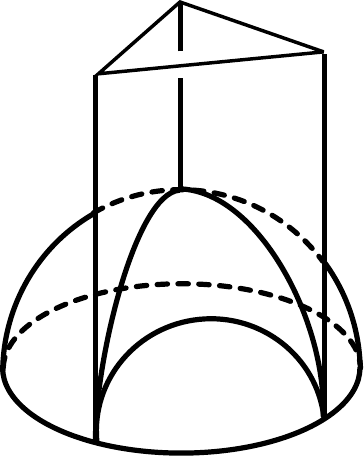}
  \caption{A $3/4$-ideal tetrahedron with ideal vertices at $\infty$, on the boundary of a hemisphere, and with finite vertex at the Euclidean maximum of that hemisphere.}
  \label{Fig:Orthoscheme}
\end{figure}

In \cite{milnor}, Milnor computed the volume of the $3/4$-ideal tetrahedron $\T$ of \reffig{Orthoscheme}. Let $\theta$ denote the dihedral angle between the two vertical faces of $\T$ that meet at the vertical edge lying over the finite vertex. By Milnor's calculation, the volume $\vol(\T)=\Lambda(\theta/2)$, where $\Lambda(\theta)$ is the Lobachevsky function:
\[
\Lambda(\theta)= - \int_{0}^{\theta} \log|2\sin t| \; dt.
\]

\begin{figure}
  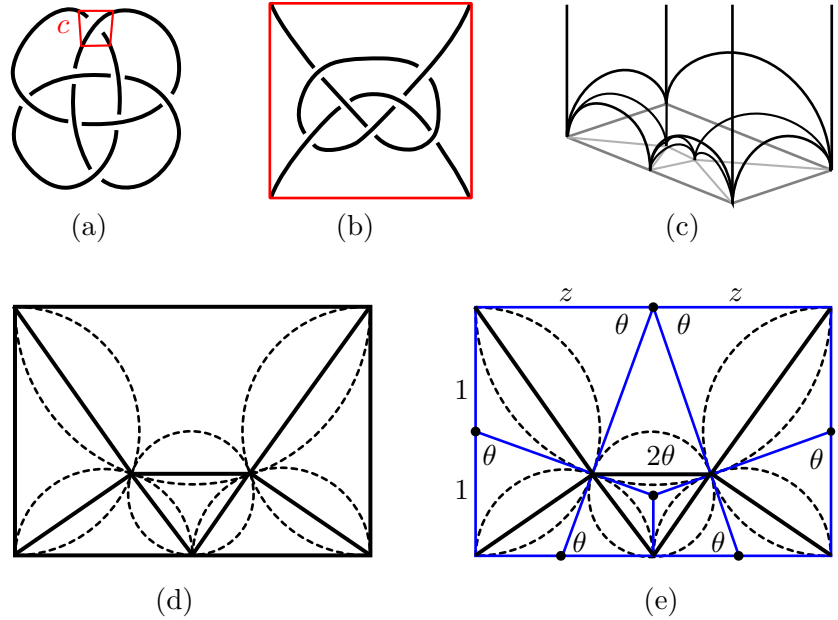
  \caption{Decomposition of the right-angled polyhedron associated to $8_{18}$:
    (a)~Crossing $c$ is chosen to be the ideal vertex at $\infty$.
    (b)~Diagram now with $c$ at infinity.
    (c)~Geometry of right-angled polyhedron, with faces in $\HH^3$ and~$c$ at~$\infty$.
    (d)~Circles making up faces of the polyhedron, lying in a rectangle.
    (e)~Circles plus projection of edges of the $3/4$-ideal tetrahedra.}  
\label{Fig:8_18}
\end{figure}

\begin{example}\label{Ex:8_18}
  We compute $\volp(L)$ for $L$ the alternating knot $8_{18}$. The process is illustrated in \reffig{8_18}.
  Note that the reduced alternating diagram of $L$ is already rationally reduced with no nontrivial 4-circuits. 

In \reffig{8_18}(a), we choose a crossing $c$, or alternatively consider $c$ as an ideal vertex of the tangle polyhedron, and take this point to infinity. This is shown in \reffig{8_18}(b).

Because the diagram is rationally reduced, we know the tangle polyhedron admits a geodesic right-angled hyperbolic structure. Each of the faces shown in \reffig{8_18}(b) will be totaly geodesic in this structure, hence faces define lines and circles on $\bdy\HH^3$. The polyhedron is shown in 3-dimensions in \reffig{8_18}(c), and circles on $\bdy\HH^3$ corresponding to the geodesic faces, viewed from the point $\infty$, are shown in \reffig{8_18}(d).

Subdivide the picture into a collection of $3/4$-ideal tetrahedra by adding a vertical edge running from the center of each circle to infinity, and adding faces from this edge to the ideal vertices. The result is shown in \reffig{8_18}(e).

Now use symmetry and trigonometry to compute the angles. In this case, six of the angles are the same value $\theta$, six others are $\pi-\theta$, one has angle $2\theta$ and one has angle $\pi-2\theta$. If we set the size of the rectangle in \reffig{8_18}(d) to be $2\times 2z$ then $\tan(\theta/2) = 1/z = z/2$. Then $\tan(\theta/2) = 1/\sqrt{2}$, and $\volp(L)\approx 12.0461$.
\end{example}

A \emph{right kite} is a kite with two right angles.
The geometric data encoded in
the rectangle $R(c)$ of \reffig{8_18}(e) is completely determined by
the tiling of $R(c)$ by right kites, shown in blue lines in
\reffig{8_18}(e).  This follows because the projection of all the
$3/4$-ideal tetrahedra to $\bdy\HH^3$ gives a rectangle tiled by
isosceles triangles, which meet in pairs across the edges of the
diagram to form right kites.  In each kite, right angles are dihedral angles of the right-angled polyhedron; the remaining kite
angles are of the form $\theta$ and $\pi-\theta$. See \reffig{Kite}.

\begin{figure}
\begingroup%
  \makeatletter%
  \providecommand\color[2][]{%
    \errmessage{(Inkscape) Color is used for the text in Inkscape, but the package 'color.sty' is not loaded}%
    \renewcommand\color[2][]{}%
  }%
  \providecommand\transparent[1]{%
    \errmessage{(Inkscape) Transparency is used (non-zero) for the text in Inkscape, but the package 'transparent.sty' is not loaded}%
    \renewcommand\transparent[1]{}%
  }%
  \providecommand\rotatebox[2]{#2}%
  \newcommand*\fsize{\dimexpr\f@size pt\relax}%
  \newcommand*\lineheight[1]{\fontsize{\fsize}{#1\fsize}\selectfont}%
  \ifx\svgwidth\undefined%
    \setlength{\unitlength}{288bp}%
    \ifx\svgscale\undefined%
      \relax%
    \else%
      \setlength{\unitlength}{\unitlength * \real{\svgscale}}%
    \fi%
  \else%
    \setlength{\unitlength}{\svgwidth}%
  \fi%
  \global\let\svgwidth\undefined%
  \global\let\svgscale\undefined%
  \makeatother%
  \begin{picture}(1,0.4375)%
    \lineheight{1}%
    \setlength\tabcolsep{0pt}%
    \put(0,0){\includegraphics[width=\unitlength,page=1]{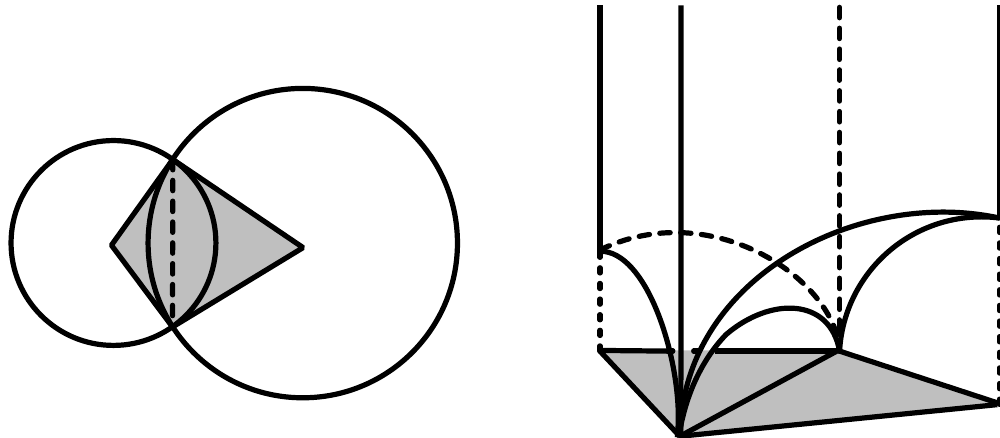}}%
    \put(0.3115731,0.18160751){\color[rgb]{0,0,0}\makebox(0,0)[lt]{\lineheight{1.25}\smash{\begin{tabular}[t]{l}$\pi-\theta$\end{tabular}}}}%
    \put(0.07832362,0.18677169){\color[rgb]{0,0,0}\makebox(0,0)[lt]{\lineheight{1.25}\smash{\begin{tabular}[t]{l}$\theta$\end{tabular}}}}%
  \end{picture}%
\endgroup%

\caption{Left: A right kite formed by radii of intersecting circles, meeting at the dashed edge $e$ shown as the short diagonal. The radii meet at right angles; the other two angles of the kite are $\theta$ and $\pi-\theta$. 
  Right: An ideal hyperbolic polyhedron is bounded by vertical planes and intersecting hemispheres above a kite, which consists of two $3/4$-ideal tetrahedra.}
\label{Fig:Kite}
\end{figure}

Therefore, we can view the procedure described in
Example~\ref{Ex:8_18} as a geometric realization problem.  Namely, the
geometric \reffig{8_18}(d) is the realization of its combinatorial
graph, shown with solid lines in \reffig{8_18}(d), which can be
obtained directly from the diagram in \reffig{8_18}(b).
We generalize this procedure below, showing that for appropriate prime
alternating links, the corresponding combinatorial graph can be
geometrically realized.

The \emph{central triangulation} of a face of a plane graph is obtained by adding a central vertex to the face, and edges joining the central vertex to all other vertices, triangulating the face. 

\begin{theorem}\label{Thm:RightAngledKites}
Let $L$ be a prime alternating link, whose link diagram is already
rationally reduced and has no nontrivial $4$-circuits.  Let
$\Gamma(L)$ be the projection graph of the link diagram.  Fix any
crossing $c$ of $L$, and let $\F(c)$ be the closure of the four faces
of $\Gamma(L)$ which meet $c$.  Let $G$ be the graph obtained by
taking the central triangulation of each face of $\Gamma(L)$ that does
not meet $c$, excluding the edges in $\Gamma(L)$. Then $G$ can be realized
as a Euclidean rectangle tiled by right kites, with one kite $k_e$ for
each edge $e$ of $\Gamma(L)$ not in $\F(c)$. Let $\theta_e$ and
$\pi-\theta_e$ denote the other kite angles of $k_e$.  Then
\[ \volp(L) = 2\sum\nolimits_{e\in \Gamma(L)\setminus\F(c)}\Lambda(\theta_e/2) + \Lambda((\pi-\theta_e)/2).
\]
\end{theorem}

We illustrate the graphs in \refthm{RightAngledKites} for the knot $8_{18}$ in \reffig{8_18-2}.

\begin{figure}
  \includegraphics{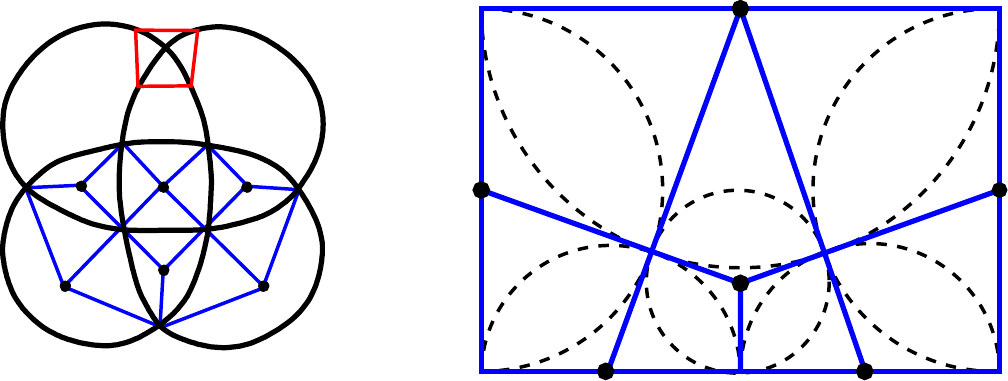}
  \caption{Left: The graph $G$, shown with blue edges, is obtained from $\Gamma(L)$, as in \refthm{RightAngledKites}, with the crossing $c$ indicated in the red box. Right: $G$ is realized as a tiling of the rectangle $R(c)$ by right kites.}  
\label{Fig:8_18-2}
\end{figure}

\begin{proof}
Let $P=P(L)$ be the checkerboard ideal polyhedron of $L$.  Since the diagram of $L$ is already rationally reduced, the $1$-skeleton of $P$ is the diagram graph of $L$, denoted $\Gamma(L)$. Hence, $P$ is $4$-valent and the number of ideal vertices of $P$ equals the number of crossings in $L$. By \reflem{Andreev}, $P$ admits a right-angled ideal hyperbolic structure.

We can realize $P$ in $\HH^3$ such that the ideal vertex corresponding to the crossing $c$ is $\infty$. The four faces incident to $c$ are vertical half-planes in $\HH^3$ intersecting at right-angles. The faces of $P$ not incident to $\infty$ are hemispheres which intersect each other and the vertical faces at right-angles.  Thus, $P$ is the polyhedron inside the chimney formed by the vertical faces and above the hemispheres, as in \reffig{8_18}(c).

Let $R(c)$ denote the rectangle in $\CC \subset \partial \HH^3$ bounded by lines which are the boundaries of vertical faces of $P$. Inside $R(c)$, an orthogonal circle pattern is formed by circles which are boundaries of the hemispherical faces of $P$. The ideal vertices of $P$, other than $\infty$, are the points of intersection of four circles, half-circles or lines in $R(c)$, as in \reffig{8_18}(d).

Let $\mathcal{G} \subset R(c)$ be the graph whose vertices are the centers of the circles in $R(c)$ and the ideal vertices of $P$ other than $\infty$, and whose edges are (Euclidean) line segments joining the center of each circle to the ideal vertices lying on that circle. Note that the boundary of the rectangle $R(c)$ consists of edges and vertices of $\mathcal{G}$.  
Since the ideal vertices of $P$, other than $\infty$, lie on the circles, every face of $\mathcal{G}$ is a quadrilateral such that two opposite vertices are centers of intersecting circles, and the other two vertices are the points of intersection. Hence every face of $\mathcal{G}$ is a kite.
The angle of intersection of two hemispheres in $\HH^3$ equals the angle of intersection of its boundary circles, which in turn equals the (equal and opposite) angles of the kite formed by the center of the circles and the points of intersection. Since $P$ is right-angled, $\mathcal{G}$ is a rectangle tiled by right kites.

On the other hand, the vertices of $G$ are the vertices of $\Gamma(L)$, other than $c$, along with the central vertex of each face of $\Gamma(L)$, other than those meeting $c$. The edges of $G$ are the edges from each central vertex to the vertices of that face. The faces of $G$ are quadrilaterals which correspond to the edges of $\Gamma(L)$ that do not meet $c$.  Each side of $R(c)$ corresponds to the boundary of each of the four faces meeting $c$. The degree of each vertical polygonal face of $P$ is the degree of the corresponding face.  Since the circles in $R(c)$ correspond to the faces of $P$ not adjacent to $\infty$, this implies that the vertices and edges of $\mathcal{G}$ and $G$ coincide, and hence they are isomorphic as planar graphs. 

The volume formula then follows from the result of Milnor~\cite{milnor}. 
\end{proof}

\subsection{Right-angled volume for weaving knots}\label{sec:Wpq}
We now apply \refthm{RightAngledKites} to an infinite family of knots
and links.  A {\em weaving knot}
$W(p,q)$ is the alternating knot or link with the same projection as
the standard closed $p$--braid $(\sigma_1\ldots\sigma_{p-1})^q$
diagram of the torus knot or link $T(p,q)$. The knot $8_{18}$ is the
weaving knot $W(3,4)$. See Table~\ref{Table1} below for several other
examples of weaving knots.  See \cite{ckp:weaving} for more details on
weaving knots.

\begin{theorem}\label{thm:Wpq}
For all weaving knots $W(3,q)$, we can compute $\volp(W(3,q))$ by an
algorithm that requires solving a one-variable polynomial equation.
\end{theorem}

\begin{proof}
\refthm{volp-tangles}\,(1) implies $\volp(W(3,q))=0$ for $q\leq 2$. To compute $\volp(W(3,q))$ for $q\geq 3$, we follow the procedure in \refthm{RightAngledKites} to obtain its right kite structure. We start with the case $W(3,2n+1)$.  See \reffig{Wpq_odd}, which shows a knot diagram for $W(3,7)$ and its right kite structure, with certain kite angles and edge lengths in a $2\times 2z$ rectangle.  For $W(3,2n+1)$, the kite angles are $\theta_i$, and the edge lengths are $z$ and $\{x_i,\, y_i\}$ for $i=1,\ldots, n-1$, extending the pattern shown in \reffig{Wpq_odd} for $n=3$. All the other angles and edge lengths can then be determined immediately from these.  Note that for all $W(3,2n+1)$, the two kites centered at the bottom are squares.

\begin{figure}
  \begin{center}
    \includegraphics{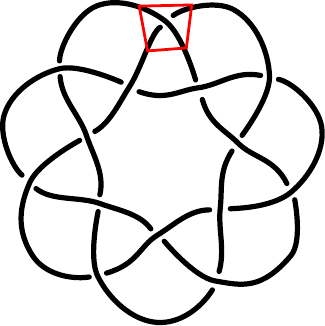} \quad
    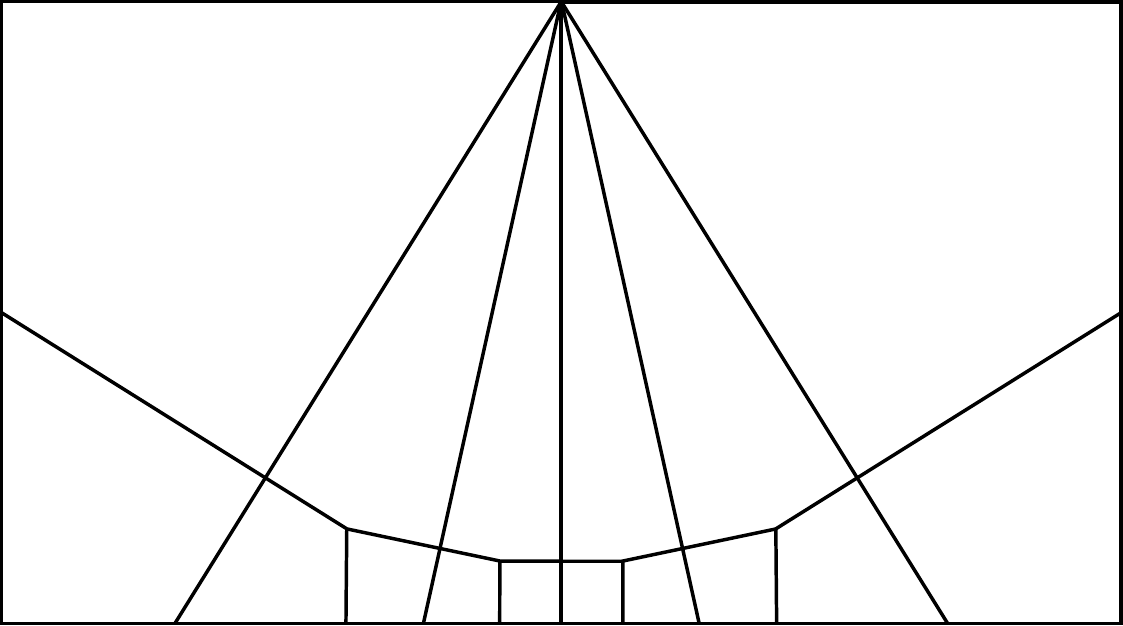
\caption{Left: $W(3,7)$ diagram with chosen crossing. Right: right kite structure for $W(3,7)$, with certain kite angles and edge lengths labeled.}
\label{Fig:Wpq_odd}
  \end{center}
\end{figure}

The formulas below are for $\{x_i, y_i, \theta_i \,|\,i=1,\ldots, n-1\}$, and we define $y_0=1$.
\[
\tan\frac{\theta_1}{2} = x_1 = \frac{y_1}{x_1} =\frac{1}{z}
\quad \mbox{hence } \quad y_1=x_1^2 \ \text{and} \ z=\frac{1}{x_1}
\]
\[
\tan\frac{\theta_i}{2} = \frac{y_i}{x_i} = \frac{x_i}{y_{i-1}}
\quad \mbox{hence } \quad x_i^2 = y_i y_{i-1}
\]
\[
  \tan\frac{\theta_{i+1}-\theta_i}{2} = \frac{y_i}{z}=x_1 y_i
\quad \mbox{hence } \quad
\frac{\tan\frac{\theta_{i+1}}{2}-\tan\frac{\theta_i}{2}}{1+\tan\frac{\theta_{i+1}}{2}\tan\frac{\theta_i}{2}} =x_1 y_i.
\]
This implies
\[  x_i y_{i+1} - x_{i+1} y_i = x_1 y_i ( x_i x_{i+1} + y_i y_{i+1}), \mbox{ thus} \]
\[  x_i (x_{i+1}^2/y_i) - x_{i+1} y_i = x_1 y_i (x_i x_{i+1} +x_{i+1}^2), \mbox{ so}\]
\[ x_{i+1} = \frac{x_1 x_i y_i^2 + y_i^2}{x_i-x_1 y_i^2}. \]
These equations imply that $x_{i+1}$ can always be expressed by a
rational function in just the variable $x_1$.  For example, for all
$W(3,q)$, $x_2 = (x_1^5+x_1^3)/(1-x_1^4)$.  In
addition, using the midline of the rectangle, $z+y_{n-1}=2$, which implies
that $\sum x_i = z-1 = 1/x_1 -1$.
This yields one polynomial equation in $x_1$, whose solution provides a
solution for all edge lengths, $x_i,\, y_i$.

For $W(3,2n)$, the right kite structure extends the pattern
shown in \reffig{8_18}(e), and we get the same equations as
above, except $\sum x_i = z/2 = 1/(2x_1)$.

Finally, the kite angles $\theta_i$ can be computed from the edge
lengths.  By \refthm{RightAngledKites}, this provides an algorithmic
solution for $\volp(W(3,q))$ for each $q$.
\end{proof}

See Table~\ref{Table1} for numerical computations of $\volp(W(p,q))$
for several small values of $p,\, q$.  The equality for the Borromean
link $L$ is exact: all four of its right kites are squares, so
$\volp(L)=16\,\Lambda(\pi/4)=\vol(L)$.  The exact equality for $W(4,4)$ was proved in \cite{Gan}.

\begin{table}[h]
\caption{Right-angled volumes for several low-crossing weaving knots}
\label{Table1}
\begin{center}
\begin{tabular}{l|c|c}
$W(p,q)=L$ & $\volp(L)$ & $\vol(L)$ \\
\hline
$W(3,2)=4_1$ & \ 0 \hspace*{0.7 cm} & \ 2.0299  \\
$W(3,3)=$ Borromean link & \ 7.3277 & \ 7.3277  \\
$W(3,4)=8_{18}$ & 12.0461 & 12.3509  \\
$W(4,3)=9_{40}$ & 14.6554 & 15.0183 \\
$W(3,5)=10_{123}$ & 16.2758 & 17.0857 \\
$W(3,6)=L12a1882$ & 19.4287 & 21.6316 \\
$W(3,7)=14a19470$ & 24.2126 & 26.0544 \\
$W(4,4)=L12a2008$ & 24.0922 & 24.0922 \\
\end{tabular}
\end{center}
\end{table}

\subsection{Asymptotically sharp volume bounds}\label{Sec:Sharp}

Our results in \cite{CKPgmax, ckp:gbal} imply that the lower bound
from right-angled volume is asymptotically sharp for many sequences of
alternating links.  Using our results above, we prove that we can
remove the ``no cycle of tangles'' condition from
Theorems~\cite[Theorem~1.4]{CKPgmax} and \cite[Theorem~6.7]{ckp:gbal}.

Let $\W$ be the infinite square weave and let $\Q$ be the triaxial link,
which are biperiodic alternating links discussed in \cite{CKPgmax,
  ckp:gbal}.  Let $W$ and $Q$ be the respective toroidally alternating
quotient links.  Let $\vtet\approx 1.01494$ and $\voct\approx 3.66386$
be the hyperbolic volumes of the regular ideal tetrahedron and the
regular ideal octahedron, respectively. Let $c(K)$ denote the crossing
number of $K$.

\begin{theorem}\label{Thm:Improved}
Let $K_n$ be any sequences of alternating hyperbolic link diagrams which satisfy {\em F{\o}lner convergence almost everywhere}, as in
\cite[Definition 6.1]{ckp:gbal}, to $\W$ and $\Q$, respectively.  Then 
\[ K_n\toF\W \implies \lim_{n\to\infty} \frac{\vol(K_n)}{c(K_n)} =
\lim_{n\to\infty} \frac{2\pi\log\det(K_n)}{c(K_n)} =
\frac{\vol((T^2\times I)-W)}{c(W)} = \voct, \]
\[ K_n\toF\Q \implies \lim_{n\to\infty} \frac{\vol(K_n)}{c(K_n)} =
\lim_{n\to\infty} \frac{2\pi\log\det(K_n)}{c(K_n)} =
\frac{\vol((T^2\times I)-Q)}{c(Q)} = \frac{10\,\vtet}{3}. \]
\end{theorem}

\begin{proof}
For links without a cycle of tangles, the proofs of
Theorems~\cite[Theorem~1.4]{CKPgmax} and \cite[Theorem~6.7]{ckp:gbal}
relied on a lower volume bound coming from the checkerboard polyhedra
of $K_n$, given an ideal right-angled hyperbolic structure.  Now, we
use instead the guts polyhedra of $K_n$, whose volume is $\volp(K_n)$,
which is a lower volume bound by \refthm{GutsVolumeBound}.  Since
neither $\W$ nor $\Q$ contains a cycle of tangles, any cycles of
tangles in $K_n$ must be in $G(K_n)-G_n$.  Thus, if we
use the guts polyhedra, the proofs are otherwise unchanged.
\end{proof}

\begin{corollary}
The lower bound, $\volp(L)$ is asymptotically sharp; i.e., 
there exists a sequence of alternating hyperbolic links
$K_n$ such that 
$$ \lim_{n\to\infty}\frac{\volp(K_n)}{\vol(K_n)}=1.$$
\end{corollary}
\begin{proof}
By \refthm{Improved}, there exist sequences of alternating links
$K_n$ and $K'_n$, such that
\[  \lim_{n\to\infty}\frac{\volp(K_n)}{c(K_n)} = \lim_{n\to\infty}\frac{\vol(K_n)}{c(K_n)} = \voct \ \text{ and } \
\lim_{n\to\infty}\frac{\volp(K'_n)}{c(K'_n)} =\lim_{n\to\infty}\frac{\vol(K'_n)}{c(K'_n)} = \frac{10\,\vtet}{3}.
\]
It follows that for any sequence $K_n$ such that $K_n\toF\W$ or $K_n\toF\Q$,
\[ \lim_{n\to\infty}\frac{\volp(K_n)}{\vol(K_n)}=1. \qedhere \]
\end{proof}

For an example of such a sequence, let $K_n=W(p_n,q_n)$ for any $p_n,\, q_n\to\infty$.  Then $K_n\toF\W$.

\subsection{Comparison to other volume bounds}\label{Sec:Discussion}

By~\cite{AgolStormThurston, Lackenby}, for any hyperbolic alternating link diagram $L$ with twist number $t(L)$,
\begin{equation}\label{ASTbound}
\frac{\voct}{2}\, (t(L)-2) \: \leq \: \vol(L). 
\end{equation}

Both this lower bound and the one we obtain from $\volp(L)$ are
equalities for the Borromean link.  In general, the bound in
(\ref{ASTbound}) seems to provide information about the volume of the
``algebraic'' part of $L$. By contrast, $\volp(L)$ gives
no information on volumes of the algebraic part, but seems to give
better estimates when the algebraic part is ``small.''  For example, by \refthm{Improved}, there
are many families of alternating links $K_n$ with
$\vol(K_n)/c(K_n)\to\voct$, such that $t(K_n)/c(K_n)\to 1$. For such
alternating links $K_n$, including weaving links, $\volp(K_n)$ will be
at least twice the lower bound in (\ref{ASTbound}).

For the examples in Table~\ref{Table1}, $\volp(L)$ beats all previous
diagrammatic lower volume bounds. However, as discussed in
Section~\ref{Sec:topology}, for any visibly algebraic link $L$, there
are no tangle polyhedra, so $\volp(L)=0$.  For example, any Montesinos
link $L$, whose hyperbolic volume can be arbitrarily large, has
$\volp(L)=0$.  Similarly, for any arborescent link $L$ with no hidden
Conway spheres, $\volp(L)=0$, but the lower bound in (\ref{ASTbound})
is non-zero.  As discussed above, there also exist arborescent links
whose alternating diagram does not admit visible Conway spheres, such
as the Borromean link $B$ for which $\volp(B)=\vol(B)$.

It would be useful to be able to combine estimates on algebraic parts of links using \eqref{ASTbound} with the volume bound of \refthm{GutsVolumeBound}. However, at this time our techniques do not allow us to combine the two arguments used to prove the different bounds.

\subsection{Right-angled links}\label{sec:right-angled-knots}
We say that a hyperbolic link $L$ is \emph{right-angled} if $S^3-L$
with the complete hyperbolic structure admits a decomposition into
ideal hyperbolic right-angled polyhedra. For example, the Whitehead
link and Borromean link are right-angled alternating links.  But
$\volp(L)=\vol(L)$ only for the Borromean link; $\volp(L)=0$ for the
Whitehead link.  Among non-alternating links, the class of fully
augmented links is right-angled \cite{CDW09, Purcell-auglink}.

Hyperbolic right-angled $3$--manifolds have interesting
properties. The right-angled decomposition gives an immersed totally
geodesic surface in the $3$--manifold arising from the faces of the
polyhedra. Another property is that their fundamental groups are
virtually special \cite{CDW09}.  In general, it seems difficult to
prove that a hyperbolic $3$--manifold is not right-angled. For
example, Calegari proved that the knot $8_{20}$ has no immersed
totally geodesic surfaces \cite[Corollary 4.6]{Calegari2006}. This
implies that $8_{20}$ is not right-angled.

\begin{conjecture}\label{Conj:knot1}
  There does not exist a right-angled knot.
\end{conjecture}

Together with \refthm{GutsVolumeBound}, Conjecture~\ref{Conj:knot1} would imply that for any knot $K$,
\[ \volp(K) \:<\: \vol(K). \]

Volume bounds for ideal right-angled polyhedra provide another
obstruction for links to be right-angled.  By~\cite{atkinson:volume},
the regular ideal octahedron has the smallest volume among all such
polyhedra.  Thus, any hyperbolic link $L$ with $\vol(L)<\voct$ cannot
be right-angled; for alternating links $L$, if $\vol(L) < 2\voct$,
then $\volp(L)=0$.

Recently, Vesnin and Egorov~\cite{VesninEgorov}
enumerated the volumes of ideal right-angled polyhedra with at most
$23$ faces.  Using their enumeration, we have verified
Conjecture~\ref{Conj:knot1} for all knots up to 11 crossings.

\begin{question}\label{Ques:strict}
  Does there exist a hyperbolic alternating link $L$, besides the Borromean link, for which $\displaystyle \volp(L) \:=\: \vol(L)$?
\end{question}

\begin{remark}
Since we posted this paper on the arXiv, Gan showed that there are
exactly three links with two totally geodesic checkerboard
surfaces~\cite{Gan}. These are the Borromean link with underlying
graph the octahedron, a $4$--component link with underlying graph the
cuboctahedron (which is $W(4,4)$ from Table~\ref{Table1}), and a $6$--component link
with underlying graph the icosidodecahedron. These are illustrated in
\cite[Figure~6]{Gan}. In particular, there are no knots whose
checkerboard surfaces are right-angled totally geodesic, resolving a special case and providing more evidence for \refconj{knot1}.
Moreover, for these three links, $\volp(L)=\vol(L)$, giving two more
examples to \refques{strict}. See \cite[Remark~3.15]{Gan}.
Thus, to extend \refthm{GutsVolumeBound}, it is reasonable to conjecture that for any hyperbolic alternating link $L$, except for the three links in \cite{Gan}, 
\[ \volp(L) \: < \: \vol(L). \]
\end{remark}

\bibliographystyle{amsplain}
\bibliography{references}

\end{document}